\documentclass[12pt]{amsart}
\usepackage{color}
\usepackage[utf8]{inputenc}
\usepackage{mathrsfs}
\usepackage{amsfonts,amsmath}
\usepackage{amssymb,latexsym}
\usepackage{bbm}
\usepackage{mathtools}
\def\F{\mathcal{F}}
\def\i{\text{i}}
\def\wh{\widehat}
\def\wt{\widetilde}
\def\xial{\xi^1_{\al}}

\def\tq{\overline{q}}
\newcommand\asth{\,\underset{H}{\ast}\,}
\newcommand\astu{\,\underset{1}{\ast}\,}
\newcommand\A{\mathbf{A}}
\newcommand\B{\mathbf{B}}
\newcommand\ubdry{u_{{\scriptstyle bdry}}}
\newcommand\uint{u_{{\scriptstyle int}}}
\newcommand\bm{\mathscr{BM}}
\newcommand\ual{ u_\al }
\newcommand\qual{ q_\al }
\newcommand\ualbdry{u_{\al,bdry}}
\newcommand\ualint{u_{\al,int}}
\DeclareMathOperator\supp{supp}
\def\alto{\xrightarrow{\al\to0}}
\def\weakalto{\xrightharpoonup{\al\to0}}
\usepackage{etoolbox}
\patchcmd{\section}{\scshape}{\bfseries\Large}{}{}
\makeatletter
\renewcommand{\@secnumfont}{\bfseries\Large}
\makeatother
\newtheorem{theorem}{Theorem}%[section]
\newtheorem{lemma}[theorem]{Lemma}%[section]
\newtheorem{definition}[theorem]{Definition}%[section]
\newtheorem{proposition}[theorem]{Proposition}%[section]
\theoremstyle{remark}
\newtheorem{remark}[theorem]{\bf Remark}%[section]
\newtheorem{claim}[theorem]{\bf Claim}
\newcommand{\R}{\mathbb{R}}
\newcommand{\real}{\mathbb{R}}
\newcommand{\I}{\mathbb{I}}
\newcommand{\hp}{\mathbb{H}}
\def\al{\alpha}
\def\PP{\mathbb{P}}
\DeclareMathOperator{\dive}{div}
\DeclareMathOperator{\curl}{curl}

\def\ep{\varepsilon}
\def\om{\omega}

\def\dt{\partial_{t}}
\newcommand\nl[2]{\|#2\|_{L^{#1}}}
\begin{document}
\title{The limit $\alpha \to 0$ of the $\alpha$-Euler equations in the half plane with no-slip boundary conditions and vortex sheet initial data}
\author{A. V. Busuioc, D. Iftimie, M. C. Lopes Filho and H. J. Nussenzveig Lopes}

\begin{abstract}
In this article we study the limit when $\alpha \to 0$ of solutions to the $\alpha$-Euler system in the half-plane, with no-slip boundary conditions, to weak solutions of the 2D incompressible Euler equations with non-negative initial vorticity in the space of bounded Radon measures in $H^{-1}$. This result extends the analysis done in \cite{busuioc_weak_2017,lopes_filho_convergence_2015}. It requires a substantially distinct approach, analogous to that used for Delort's Theorem, and a new detailed investigation of the relation between (no-slip) filtered velocity and potential vorticity in the half-plane.
\end{abstract}

\maketitle

\section{Introduction}

This article concerns the limit $\alpha \to 0$ of the $\alpha$-Euler equations in the half-plane, with no-slip boundary conditions, with initial velocity in $L^2$ and initial vorticity whose singular part is a non-negative bounded Radon measure. The present work is a natural continuation of research contained in \cite{busuioc_weak_2017,lopes_filho_convergence_2015}, where the respective authors proved convergence, first for initial velocity in $H^3$, see \cite{lopes_filho_convergence_2015} and then for initial vorticity in $L^p$, $p>1$, see \cite{busuioc_weak_2017}, both for flows in bounded, smooth domains. The extension to initial vorticities in the space of Radon measures requires a substantial change in technique. The previous results are based on energy estimates and boundary correctors \cite{lopes_filho_convergence_2015} or on the compactness of the velocity sequences obtained from boundedness of the corresponding vorticity in a suitable space  \cite{busuioc_weak_2017}. For the present work, a compensated compactness argument is required, involving a subtle cancellation property of the nonlinearity, in the spirit of Delort's celebrated existence result, see \cite{delort_existence_1991-1}. Let us mention that the limit $\alpha \to 0$ for initial vorticities in the space of Radon measures, in the case of the full plane, was first considered in \cite{bardos_global_2010} and the proof was completed in \cite{gotoda_convergence_2018}. However, the presence of boundaries is a significant complication. Our work involves a detailed study of the influence of the boundary on the solution of the $\alpha$-Euler equations, the key novelty of our result. 

More precisely, much of our analysis focuses on the fine properties of the operator $\B$, introduced in Definition \ref{soloperator}, which maps the potential vorticity $q$ to the filtered velocity $u$. This is a classical pseudo-differential operator of order $-3$, given by $\B = (\mathbb{I} + \alpha {\bf A})^{-1} K_{\mathbb{H}}$, where ${\bf A}$ is the half-plane Stokes operator with no-slip boundary conditions and $K_{\mathbb{H}}$ is the Biot-Savart operator for the half-plane. It decomposes naturally into an interior part, which is easy to understand, and a boundary part, similar to a Poisson integral, which is more delicate. The analysis of the boundary part makes use of Fourier methods, one of the main reasons why we restrict ourselves to half-plane flows.

From a broader point-of-view, the $\alpha$-Euler equations are a regularization of the Euler equations, obtained by averaging the transporting velocity at scale $\sqrt{\alpha}$. It is the inviscid limit of the second-grade fluid model, see \cite{cioranescu_existence_1984-1}, the equation for geodesics in the group of volume-preserving diffeomorphisms with a natural metric, see \cite{marsden_geometry_2000} and a variant of the vortex blob method, a standard numerical method for discretizing 2D inviscid flows. The desingularized velocity is obtained from the physical one by inverting the elliptic operator $(\mathbb{I} - \alpha \PP \Delta)$, which, in a domain with boundary, requires boundary conditions. The no-slip boundary conditions are the most natural, but Navier-type conditions have also been used (see \cite{busuioc_incompressible_2012,busuioc_uniform_2016}). Choosing no-slip makes the vanishing $\alpha$ problem resemble the vanishing viscosity limit, an important open problem. In this setting, the vanishing $\alpha$ limit could present some of the complications of the vanishing viscosity limit, such as boundary layers and spontaneous small-scale generation, see \cite{bardos_mathematics_2013-1}. This similarity between the present problem and vanishing viscosity is the chief motivation for the present work. The results obtained to date, including those we present here, suggest that these two limits behave in sharply distinct ways, but it is not entirely clear why that might be the case.

The remainder of this article is organized as follows. Still in the Introduction, we briefly state our main results. In Section \ref{sect-notandprelim} we fix notation, we introduce elementary facts of Potential Theory in the half-plane and compute some Fourier transforms. In Section \ref{sect-green} we introduce the operator $\B$, which maps potential vorticity to filtered velocity. In Section \ref{sect-existence} we sketch the proof of Theorem \ref{theo-existence}, the existence result for $\alpha$ fixed. This is an adaptation to the case of the half-plane of a similar result in the full-plane case, see \cite{oliver_vortex_2001}. In Section \ref{sect-intbdry} we introduce the decomposition of the operator $\B$ in interior and boundary parts. In Section \ref{sect-l1est} we derive precise estimates for the boundary potentials associated with the operator $\B$. In Section \ref{sect-passlim} we apply the results obtained to prove Theorem \ref{theo-convergence}, adapting Schochet's argument, see \cite{schochet_weak_1995}, and the argument used in \cite{lopes_filho_existence_2001}. Finally, in the last section, we present some concluding remarks and a few open problems.

Let us continue with some notation. We will denote by $\hp$ the half-plane
\begin{equation*}
  \hp=\{x\in\R^2\ ;\ x_2>0\}.
\end{equation*}

The initial-value problem for the $\alpha$-Euler equations with no-slip boundary conditions on $\hp$ are given by:
\begin{equation*} 
\left\{\begin{array}{ll}
\partial_t (u - \alpha \Delta u) + u\cdot\nabla (u - \alpha \Delta u) + \sum_j (u - \alpha \Delta u)_j\nabla u_j =-\nabla p, & \text{ in } \real_+ \times \hp\\
& \\
\dive u = 0, &   \text{ in } \real_+ \times \hp,\\
& \\
u = 0, & \text{ on } \real_+ \times \partial \hp,\\
& \\
u(0,\cdot) = u_0, & \text{ on } \{t=0\} \times \hp.
\end{array}\right.
\end{equation*}
Above, $u - \alpha \Delta u$ is called the {\it unfiltered} velocity, $u$ is the {\it filtered} velocity and $p$ is the scalar pressure.

Taking the curl of the $\alpha$-Euler equations, in two dimensions, gives rise to an active scalar transport equation  given by
\begin{equation} \label{potvorteq}
\left\{
\begin{array}{l}
\partial_t q+u\cdot\nabla q=0,\\
q(0,\cdot)=q_0,
\end{array}
\right.
\end{equation}
where $u$ is related to the  potential vorticity $q$ through the following system:
\begin{equation} \label{velpotvort}
\left\{
\begin{array}{l}
\dive u=0,\\
\curl(u-\al\Delta u)=q\\
u\bigl|_{\partial \hp }=0.
\end{array}
\right.
\end{equation}
The scalar quantity $q \equiv \curl(\I - \alpha \Delta) u$ is called the {\it potential vorticity} associated to the velocity $u$. The equations in \eqref{potvorteq}-\eqref{velpotvort} are the potential vorticity equations, i.e. the vorticity formulation of the $\alpha$-Euler equations.

Let $\bm(\hp)$ be the set of bounded Radon measures on $\hp$ and recall that the norm of a measure $\mu$ in  $\bm(\hp)$ is given by the total variation $|\mu|(\hp)$. Set $\dot{H}^{-1}(\hp)=\{\curl w \; | \; w\in L^2(\hp)^2\}$, which we note in passing is a proper subset of $H^{-1}(\hp)$. We have that $\dot{H}^{-1}(\hp)$ is a Banach space with the norm
$\| q \|_{\dot H^{-1}} = \inf\{\|w\|_{L^2} \; | \; q = \curl w \}$. Let $\PP$ denote the Leray projector in $L^2(\hp)$ onto divergence free vector fields which are tangent to the boundary of $\hp$. Note that $\|q\|_{\dot{H}^{-1}} \equiv \| \PP w \|_{L^2(\hp)}$, independently of $w \in L^2$ such that $\curl w = q$.

We will now state our main results.

\begin{theorem}[Existence]\label{theo-existence}
Assume that $q_0\in\bm(\hp)\cap \dot{H}^{-1}(\hp)$. Then there exists a global solution $u\in C^0_b(\R_+;H^1_w(\hp))\cap C^0(\R_+\times \hp)$, $q\in L^\infty(\R_+;\bm(\hp))$ of the $\al$-Euler equations with initial data $q_0$. In addition we have the energy inequality
\begin{equation}\label{ubound}
\nl2{u(t)}^2+\al\nl2{\nabla u(t)}^2\leq \nl2{u_0}^2+\al\nl2{\nabla u_0}^2\quad \forall t\geq0.
\end{equation}
and the bound
\begin{equation}\label{qbound}
\|q\|_{L^\infty(\R_+;\bm(\hp))}\leq \|q_0\|_{\bm(\hp)}.
\end{equation}
\end{theorem}
Above $C^0_b$ is the space of bounded continuous functions and $H^1_w$ denotes the space $H^1$ endowed with the weak topology.

\begin{theorem}[Convergence]\label{theo-convergence}
Assume that $q_0\in\bigl(\bm_+(\hp)+L^1(\hp)\bigr)\cap \dot{H}^{-1}(\hp)$ is independent of $\al$. Let $u_\al$, $q_\al$ be a global solution of the $\al$--Euler equations with initial data $q_0$, as obtained in Theorem \ref{theo-existence}. Then there exists a vortex sheet solution $v\in L^\infty(\R_+;L^2(\hp))$, with $\om=\curl v\in L^\infty(\R_+;\bm(\hp))$, of the incompressible Euler equations with initial vorticity $\om_0=q_0$, and a subsequence $u_{\al_k}$, $q_{\al_k}$ such that $u_{\al_k}\rightharpoonup v$  weak-$\ast$ $L^\infty(\R_+;L^2)$  and $q_{\al_k}\rightharpoonup \om$  weak-$\ast$ $L^\infty(\R_+;\bm(\hp))$.
\end{theorem}

\section{Notations and some preliminary results} \label{sect-notandprelim}

We begin by fixing notation. The constant $C$ denotes a generic constant whose value may change from one line to another. If $(a,b) \in \real^2$ then we denote  $(a,b)^\perp \equiv (-b,a)$.

We will use standard notation for function spaces: $L^p$ (Lebesgue space), $W^{m,p}$ (Sobolev space), $L^{2,\infty}$ (Lorentz space), $\bm$ (bounded Radon measures), etc. All function spaces are defined on $\hp$ unless otherwise specified. The notation $L^p_\sigma$ denotes the space of $L^p$ divergence free vector fields tangent to the boundary endowed with the $L^p$ norm. We define in a similar manner $L^{2,\infty}_\sigma$.

Recall the Leray projector $\PP$, \textit{i.e.} the $L^2$ orthogonal projector from $L^2$ to $L^2_\sigma$. It is well-known that $\PP$ can be extended to a bounded operator from $L^p$ to $L^p_\sigma$ for all $1<p<\infty$. The Stokes operator $\A$ is defined as $\A=-\PP\Delta$. Various regularity properties for the Stokes operator in $L^p$ spaces on a half-plane were proved in \cite{borchers_$l^2$_1988}. We note, in particular, that, for any $\al > 0$, $\mathbb{I}+\al\A$ is invertible on $L^p_\sigma(\hp)$ with values in $W^{2,p}(\hp)\cap W^{1,p}_0(\hp)$, see Section 3 of \cite{borchers_$l^2$_1988}. We denote this
inverse by $(\I + \al \A)^{-1}$.

Let us start with a very simple $H^1$ estimate.

\begin{lemma}\label{lem-stokes}
Let $f\in L^2$. Then $u=(\I+\al\A)^{-1}\PP f\in H^2\cap H^1_0$ and we have the following estimate:
\begin{equation*}
\nl2{u}^2+\al\nl2{\nabla u}^2\leq \nl2{f}^2.
\end{equation*}
In particular, the operator $(\I + \al \A)^{-1}\PP$ is continuous from $L^2$ to $H^1_0$.
\end{lemma}
\begin{proof}
We know from the results of \cite{borchers_$l^2$_1988} that $u\in H^2\cap H^1_0$.
We have that
\begin{equation*}
u-\al\Delta u+\nabla\pi=f
\end{equation*}
for some $\pi$.
We multiply the above relation by $u$ and integrate by parts using that $u$ is divergence free and vanishes at the boundary. We get that
\begin{equation*}
\nl2{u}^2+\al\nl2{\nabla u}^2=\int f\cdot u\leq \nl2{f}\nl2{u}\leq \nl2{f}(\nl2{u}^2+\al\nl2{\nabla u}^2)^{\frac12}
\end{equation*}
so
\begin{equation*}
\nl2{u}^2+\al\nl2{\nabla u}^2\leq \nl2{f}^2.
\end{equation*}
\end{proof}

 The Fourier transform in $\R^2$ is denoted by $\F$:
 \begin{equation*}
   \F(f)(\xi)=\int_{\R^2}e^{-i x\cdot\xi}f(x)\,dx
 \end{equation*}

The Fourier transform in $\R$ is denoted by $\F_\R$ or $\wh g=\F_\R g$ where $g$ is defined on $\R$:
\begin{equation*}
\F_\R g(s)=\widehat g(s)=\int_\R e^{-ist}g(t)\,dt.
\end{equation*}

 For functions of two variables we will use the partial Fourier transform in the first variable and we will denote it by $\F_1$, or also $\F_1 f=\widetilde f$. That is, for functions $f$ defined on $\R^2$ or on $\hp$ we define
 \begin{equation} \label{tildenotation}
    \F_1(f)(\xi_1,x_2)= \widetilde{f}(\xi_1,x_2)
 \equiv \int_{\R}e^{-i x_1\cdot\xi_1}f(x_1,x_2)\,dx_1
 \end{equation}
We define in the same manner $\F_2$ the partial Fourier transform in the second variable.

We denote by $G_\al$ the Green's function of the operator $\mathbb{I}-\al\Delta$ in $\R^2$, \textit{i.e.}
\begin{equation}\label{Gal}
  G_\al(x)=\F^{-1}\bigl(\frac1{1+\al|\xi|^2}\bigr).
\end{equation}
We have that
\begin{equation}\label{galphag1}
  G_\al(x)=\frac1{\al}G_1\bigl(\frac x{\sqrt\al}\bigr)
\end{equation}
where
\begin{equation*}
  G_1(x)=\F^{-1}\bigl(\frac1{1+|\xi|^2}\bigr).
\end{equation*}
is a function who is exponentially decaying at infinity and has a logarithmic singularity at the origin.

The Green's function of the Laplacian in $\R^2$ is denoted by
\begin{equation*}
G(x)=\frac1{2\pi}\ln|x|.
\end{equation*}

We shall also use the following function
\begin{equation}\label{defhalpha}
  H_\al(x)=\al G_\al(x)+G(x)=\al G_\al(x)+\frac1{2\pi}\ln|x|.
\end{equation}

A scalar function $\omega \in L^p(\real^2)$ gives rise to a divergence-free vector field $u$ on $\real^2$ whose $\curl$ is $\omega$ through the Biot-Savart law: $u = K \ast \omega$, with the (Biot-Savart) kernel  $K$, given by:
\begin{equation} \label{defK}
K(x)=\nabla^\perp G(x)=\frac{x^\perp}{2\pi|x|^2}.
\end{equation}

We also need to introduce the following smoothed-out kernel
\begin{equation} \label{defkal}
K_\al=K\ast G_\al.
\end{equation}

\medskip

We recall now several well-known (inverse) Fourier transforms. For all $a>0$ we have that
\begin{equation}
\label{four0}
\F_\R\bigl(\frac1{t^2+a^2}\bigr)=\frac\pi{a}e^{-a|s|},
\end{equation}
\begin{equation}\label{four1}
\F_\R^{-1}\bigl(\frac1{s^2+a^2}\bigr)=\frac1{2a}e^{-a|t|}
\end{equation}
and
\begin{equation}\label{four2}
\F_\R^{-1}(e^{-a|s|})=\frac{a}{\pi(t^2+a^2)}.
\end{equation}
Differentiating with respect to $a$ the relation above also yields the following inverse Fourier transform:
\begin{equation}\label{four4}
\F_\R^{-1}(|s| e^{-a|s|})=\frac{a^2-t^2}{\pi(t^2+a^2)^2}.
\end{equation}

Applying $\F_1$ to \eqref{Gal} and using \eqref{four1} we get that
\begin{equation*}
  \F_1 G_\al(\xi_1,x_2)=\F_2^{-1}\bigl(\frac1{1+\al|\xi|^2}\bigr)(\xi_1,x_2)=\frac1{2\al\xial}e^{-\xial |x_2|}.
\end{equation*}
where we used the notation
\begin{equation*}
  \xial=\sqrt{\xi_1^2+\frac1\al}.
\end{equation*}
Differentiating with respect to $x_2$ yields
\begin{equation}\label{g1}
  \F_1 \partial_2G_\al(\xi_1,x_2)=-\frac1{2\al}e^{-\xial x_2}\qquad\forall x_2>0.
\end{equation}

Using \eqref{four0} we compute
\begin{equation}\label{g2}
\F_1 \partial_2G(\xi_1,x_2)=\frac{x_2}{2\pi}\F_1\bigl(\frac1{x_1^2+x_2^2}\bigr)=\frac12 e^{-x_2|\xi_1|}\qquad\forall x_2>0.
\end{equation}

From \eqref{defhalpha}, \eqref{g1} and \eqref{g2} we conclude that for all $x_2>0$ we have that
\begin{equation}\label{fhalpha}
\F_1 \partial_2H_\al(\xi_1,x_2)=\al\F_1 \partial_2G_\al(\xi_1,x_2)+
\F_1 \partial_2G(\xi_1,x_2)
=\frac12 (e^{-x_2|\xi_1|}-e^{-x_2 \xial }).
\end{equation}

In the remainder of this work we will frequently need to consider the odd extension to $\real^2$ of a scalar defined on $\hp$, as well as the Biot-Savart law induced by the extension. To this end we introduce the following notation: if $q \in C^\infty(\hp)$ then its odd extension will be denoted $\tq=\tq(x)$ and is given by

\begin{equation*} 
\tq(x)=\begin{cases}
    q(x) & \text{if } x_2>0\\
   -q(x_1,-x_2) & \text{if } x_2<0.
  \end{cases}
\end{equation*}

If $q \in \bm(\hp)$ then we will also need to consider its odd extension, still denoted $\tq$. Let $\varphi \in C^0_0(\real^2)$ be a test function and split $\varphi$ into its odd and even parts: $\varphi = \varphi_o + \varphi_e$. Then $\langle \tq, \varphi \rangle \equiv 2\langle q, \varphi_o \rangle $. Of course $\tq \in \bm(\real^2)$. Furthermore, if $q \in \dot{H}^{-1}(\hp)$ then $\tq \in \dot{H}^{-1}(\real^2)=\{\curl w \; | \; w \in L^2(\real^2)\}$.

Let us consider the Biot-Savart velocity field in $\real^2$ induced by the odd extension of a measure $q \in \bm(\hp)$, namely $K \ast \tq$, where $K$ was introduced in \eqref{defK}. The odd symmetry of $\tq$ induces a covariant symmetry under which the first component $(K\ast\tq)^1$ is even while the second component $(K\ast\tq)^2$ is odd, with respect to $x_2$. Hence $K\ast\tq\bigl|_{\hp}$ is divergence free, tangent to the boundary of $\hp$, and its  curl in $\hp$ is $q$. This vector field can be written as
\begin{equation*}
K\ast\tq\bigl|_\hp=\int_{\R^2}K(x-y)\,d\tq(y)=\int_\hp K_\hp(x,y)\,dq(y),
\end{equation*}
where
\begin{equation} \label{khalfplane}
K_\hp(x,y)\equiv \frac{(x-y)^\perp}{2\pi|x-y|^2}-\frac{(x-\overline y)^\perp}{2\pi|x-\overline y|^2}
\end{equation}
where $\overline y=(y_1,-y_2)$ is the image of $y$.

The Biot-Savart law in the half-plane $\hp$ is the integral operator acting on $q$ given by
\begin{equation} \label{defBS}
K_\hp[q] = \int_{\hp} K_\hp(x,y)\, dq(y),
\end{equation}
and  $K_\hp$ is the Biot-Savart kernel in the half-plane. Note that $K_\hp[q]=K\ast\tq\big|_\hp$.

\section{Finding velocity from potential vorticity: the solution operator}\label{sect-green}
Let $u$ be the velocity in the half-plane induced by a potential vorticity $q$, the solution of the system of equations \eqref{velpotvort}. The aim of this section is to produce and understand the solution operator for  $u$ in terms of $q$.

We begin with an estimate for the Biot-Savart law induced by a measure $q \in \bm(\hp)$.

\begin{lemma}\label{lem-kh}
Let $q\in\bm(\hp)$ and consider $K_\hp[q]$ as introduced in \eqref{defBS}.
Then $K_\hp[q]\in L^{2,\infty}_\sigma$ and $\|K_\hp[q]\|_{L^{2,\infty}}\leq C\|q\|_{\bm(\hp)}$ for some universal constant $C>0$.

If, in addition, $q\in \dot{H}^{-1}(\hp)$ then $K_\hp[q]\in L^2_\sigma$ and $\|K_\hp[q]\|_{L^2}=\|q\|_{\dot{H}^{-1}(\hp)}$.
\end{lemma}

\begin{proof}
Recall that, since $q\in\bm(\hp)$, its odd extension $\tq$ belongs to $\bm(\R^2)$. Since $K\in L^{2,\infty}(\R^2)$ it follows from the Young inequality in Lorentz spaces that $K\ast\tq\in  L^{2,\infty}(\R^2)$. Therefore $K_\hp[q] \in L^{2,\infty}_\sigma(\hp)$ and the estimate follows from said Young's inequality.

Assume, additionally, that $q\in \dot{H}^{-1}(\hp)$. Then $q=\curl w$ with $w\in L^2$. From the properties of the Leray projector $\PP$, we know that $\PP w\in L^2_\sigma$, i.e. $\PP w$ is divergence free and tangent to the boundary of $\hp$. Furthermore $w-\PP w$ is a gradient, so that $\curl \PP w = q$. We infer that $K_\hp[q]=\PP w\in L^2$ and $\|K_\hp[q]\|_{L^2} = \|\PP w\|_{L^2} = \|q\|_{\dot{H}^{-1}(\hp)}$. This completes the proof of the lemma.
\end{proof}

We claim that the solution of \eqref{velpotvort}, $u$, satisfies
\begin{equation*}
(\I+\al\A) u=K_\hp[q],
\end{equation*}
where $\A=-\PP\Delta$ is the Stokes operator.

Indeed, we have that
\begin{equation*}
\curl[(\mathbb{I}-\al\Delta)u]=q=\curl K_\hp[q]
\end{equation*}
so $(\mathbb{I}-\al\Delta)u$ and $ K_\hp[q]$ differ by a gradient. Since the Leray projection vanishes for gradient fields and reduces to the identity on divergence free vector fields which are tangent to the boundary, it follows that
\begin{equation*}
K_\hp[q]=\PP K_\hp[q]=\PP[(\mathbb{I}-\al\Delta)u]=\PP u-\al\PP\Delta u=u+\al\A u,
\end{equation*}
as desired.

Since $u$ vanishes on the boundary of $\hp$ we find that
\begin{equation}
u=(\I+\al\A )^{-1}K_\hp[q].
\end{equation}

From Lemma \ref{lem-kh} we have that, if $q \in \dot{H}^{-1}(\hp)$ then $K_\hp[q] \in L^2(\hp)$. Hence, from Lemma \ref{lem-stokes}, it follows that $u \in L^2(\hp)$.

\begin{definition} \label{soloperator}
The solution operator for the system of equations \eqref{velpotvort}, denoted $\B=\B(q)$, is given by
\begin{eqnarray*}
    \B: & \dot{H}^{-1}(\hp) & \to       L^2(\hp)  \\
        &         q         & \mapsto   (\I+\al\A )^{-1}K_\hp[q].
\end{eqnarray*}
\end{definition}

In view of Lemma \ref{lem-stokes} we actually have  $\B(q) \in H^2\cap H^1_0$, if $q \in \dot{H}^{-1}(\hp)$, and that $\B$ is continuous from $\dot{H}^{-1}$ to $H^1_0$. With this notation the solution of \eqref{velpotvort} is $u=\B(q)$.

\section{The existence result}
\label{sect-existence}

In this section we will establish Theorem \ref{theo-existence}. The strategy is standard, so we will only give a sketch of the proof. We emphasize that, in this section, $\alpha > 0$ is fixed.

\begin{proof}[Proof of Theorem \ref{theo-existence}]
Fix $q_0 \in \bm(\hp) \cap \dot H^{-1}$. We wish to obtain a solution of the $\al-$Euler equations with initial potential vorticity $q_0$. The strategy will be to choose a sequence of smooth approximations $q_0^n$ to $q_0$ and solve the $\al-$Euler equations with $q_0^n$ as initial potential vorticities. This results in a sequence of smooth, time-dependent, potential vorticities $q^n$ and vector fields $u^n$ which we subsequently pass to a weak limit. We will provide sufficient estimates for $q^n$ and $u^n$ to show that such a weak limit is a weak solution of the $\al-$Euler equations and has, as initial potential vorticity, $q_0$.

The initial velocity $u_0$ is determined through the solution operator: $u_0 = \B(q_0)$, where $\B$ was introduced in  Definition \ref{soloperator}. Furthermore, since  $q_0\in \dot H^{-1}$ we have, using Lemma \ref{lem-stokes}, that  $u_0 \in H^2 \cap H^1_0$.

We begin by constructing the smooth approximations. Let $\tq_0$ be the odd extension of $q_0$ to the full plane as described in Section \ref{sect-notandprelim}; of course $\tq_0 \in \bm(\real^2)$ and, also, $\tq_0 \in \dot{H}^{-1}(\real^2)$. Consider the Biot-Savart law in $\real^2$ for $\tq_0$: $K \ast \tq_0$, where $K$ is the Biot-Savart kernel given in \eqref{defK}. Since $\tq_0\in \dot H^{-1}$ we have that $K \ast \tq_0 \in L^2(\R^2)$. We mollify $\tq_0$ with an even, smooth, compactly supported mollifier $\phi_n$, and we denote by $\tq^n_0=\tq_0\ast\phi_n$ the resulting mollified potential vorticity. Let $q^n_0=\tq^n_0\bigl|_\hp$. Clearly we have that
\begin{equation*}
\nl1{q^n_0}\leq\|q_0\|_{\bm(\hp)}, \,  \|q^n_0\|_{\dot H^{-1}} \leq\|q_0\|_{\dot H^{-1}}.
\end{equation*}
In addition, since clearly we have $\tq_0^n \rightharpoonup \tq_0 \text{ weak}-\ast \bm(\real^2)$, it follows that
\begin{equation*}
  q_0^n \rightharpoonup q_0 \text{ weak}-\ast \bm(\hp).
\end{equation*}
By construction we have $K\ast \tq_0^n \to K \ast \tq_0$ strongly in $L^2(\real^2)$. Therefore
\begin{equation} \label{f0n2f0L2}
K_\hp[q_0^n]  \to K_\hp[q_0]  \text{ strongly in } L^2(\hp).
\end{equation}

Let $u_0^n = \B(q_0^n)$. Using \eqref{f0n2f0L2} and Lemma \ref{lem-stokes} we find that
\begin{equation}\label{convl2}
\nl2{u^n_0}^2+\al\nl2{\nabla u^n_0}^2\to\nl2{u_0}^2+\al\nl2{\nabla u_0}^2
\end{equation}
as $n\to\infty$.

Because $\tq_0 \in \dot{H}^{-1}(\real^2)$ we deduce that, for each fixed $n$, $\tq_0^n \in L^2(\real^2)$. Hence $q^n_0\in L^2$ and therefore we have that $K_\hp[q^n_0]\in H^1\cap L^2_\sigma$, so that
$u_0^n= (\I+\al\A)^{-1}K_\hp[q^n_0]$ belongs to $H^3\cap H^1_0$. For such initial data there exists a global solution $u^n$ in $H^3$ of the $\al$--Euler equations having $u_0^n$ as initial velocity. To see this one can use, for instance the method employed in \cite{lopes_filho_convergence_2015}. Since $q^n=\curl(u^n-\al\Delta u^n)$ is transported by $u^n$ and $\dive u^n = 0$, it follows that $\|q^n(\cdot,t)\|_{L^1}$ is a conserved quantity. Because $q^n(0)=q^n_0$ is bounded in $L^1$ we infer that $q^n$ is bounded in $L^\infty(\R_+;L^1)$ and
\begin{equation}\label{qnbound}
\|q^n\|_{L^\infty(\R_+;L^1)}=\nl1{q^n_0}\leq\|q_0\|_{\bm(\hp)}.
\end{equation}
In addition, if we multiply the equation of $u^n$ by $u^n$ and integrate by parts we get the classical energy equality:
\begin{equation}\label{unbound}
\nl2{u^n(t)}^2+\al\nl2{\nabla u^n(t)}^2=\nl2{u^n(0)}^2+\al\nl2{\nabla u^n(0)}^2
\qquad\forall t\geq0.
\end{equation}
In particular, $u^n$ is bounded in $L^\infty(\R_+;H^1)$.

Given the boundedness of $u^n$ in $L^\infty(\R_+;H^1)$ and of $q^n$ in $L^\infty(\R_+;L^1)$, we can extract subsequences relabeled $u^n$ and $q^n$ which converge
\begin{equation*}
u^n\rightharpoonup u \quad\text{in }L^\infty(\R_+;H^1)\text{ weak}\ast
\end{equation*}
and
\begin{equation}\label{qnweakconv}
q^n\rightharpoonup q \quad\text{in }L^\infty(\R_+;\bm(\hp))\text{ weak}\ast.
\end{equation}

As the equations are nonlinear, this convergence is not sufficient to pass to the limit in the weak form of the $\al$-Euler equations. We now proceed as follows. Let $p\in(1,2)$ and $q\in(2,\infty)$ be  dual indexes: $\frac1p+\frac1q=1$. We will prove that $\partial_t u^n$ is bounded in $L^p$ independently of $n$. Let $\varphi\in L^q_\sigma$. Recall that $\A=-\PP\Delta$ is the Stokes operator. Set $\psi=(\I+\al\A)^{-1}\varphi$. Then $\psi \in W^{2,q}\cap W^{1,q}_0\cap L^q_\sigma$ and $\|\psi\|_{W^{2,q}}\leq C\nl q\varphi$ (see \cite{borchers_$l^2$_1988}).

The velocity formulation of the $\al$--Euler equations can be written as follows (see \cite{iftimie_remarques_2002}):
\begin{multline*} 
\dt (u^n-\al\Delta u^n)+u^n\cdot\nabla u^n
-\al\sum_{j,k}\partial_j\partial_k(u^n_j\partial_k u^n)
+\al\sum_{j,k}\partial_j(\partial_k u^n_j\partial_k u^n)\\
 -\al\sum_{j,k}\partial_k(\partial_k u^n_j\nabla u^n_j)
+\nabla p^n=0.
\end{multline*}

We apply the Leray projector $\PP$ above and multiply by $\psi$. We get
\begin{equation} \label{lasteq}
\begin{aligned}
\langle \partial_t(u^n+\al\A u^n),\psi\rangle&= -\langle  \PP\bigl[u^n\cdot\nabla u^n
-\al\sum_{j,k}\partial_j\partial_k(u^n_j\partial_k u^n)
+\al\sum_{j,k}\partial_j(\partial_k u^n_j\partial_k u^n)\\
&\hskip 6cm  -\al\sum_{j,k}\partial_k(\partial_k u^n_j\nabla u^n_j)\bigr],\psi\rangle\\
&=- \langle  u^n\cdot\nabla u^n
-\al\sum_{j,k}\partial_j\partial_k(u^n_j\partial_k u^n)
+\al\sum_{j,k}\partial_j(\partial_k u^n_j\partial_k u^n)\\
&\hskip 6cm  -\al\sum_{j,k}\partial_k(\partial_k u^n_j\nabla u^n_j),\psi\rangle.
\end{aligned}
\end{equation}

Observe that
\begin{equation*}
\langle \partial_t(u^n+\al\A u^n),\psi\rangle
=\langle \partial_t u^n,(\I+\al\A)\psi\rangle
=\langle \partial_t u^n,\varphi\rangle.
\end{equation*}
Next we estimate the first of the four terms in \eqref{lasteq}:
\begin{equation*}
|\langle  u^n\cdot\nabla u^n,\psi\rangle|=\bigl|\int_\hp  u^n\cdot\nabla u^n\cdot\psi\bigl|\leq \|u^n\|_{H^1}^2\nl\infty\psi\leq C\|u^n\|_{H^1}^2\|\psi\|_{W^{2,q}}\leq C\|u^n\|_{H^1}^2\nl q\varphi.
\end{equation*}

To estimate the second and third terms in \eqref{lasteq} we integrate by parts, using that $u^n$ and $\psi$ vanish at the boundary. We deduce that:
\begin{multline*}
  |\langle  \al\partial_j\partial_k(u^n_j\partial_k u^n),\psi\rangle|
=\bigl|\al\int\partial_j\partial_k(u^n_j\partial_k u^n)\cdot\psi\bigr|
=\bigl|\al\int u^n_j\partial_k u^n\cdot\partial_j\partial_k\psi\bigr|\\
\leq C\nl{\frac{2q}{q-2}}{u^n}\nl2{\partial_k u^n}\nl q {\partial_j\partial_k\psi}
\leq C\|u^n\|_{H^1}^2\nl q\varphi
\end{multline*}
and
\begin{multline*}
  |\langle  \al\partial_j(\partial_k u^n_j\partial_k u^n),\psi\rangle|
=\bigl|\al\int\partial_k u^n_j\partial_k u^n\cdot\partial_j\psi\bigr|
\leq \al\|u^n\|_{H^1}^2\nl\infty{\nabla\psi} \\
\leq C\|u^n\|_{H^1}^2\|\psi\|_{W^{2,q}}
\leq C\|u^n\|_{H^1}^2\nl q\varphi.
\end{multline*}
The same estimate holds true for the last term in \eqref{lasteq}. We have thus obtained the following bound
\begin{equation*}
|\langle \partial_t u^n,\varphi\rangle|\leq C\|u^n\|_{H^1}^2\nl q\varphi\quad\forall\varphi\in L^q_\sigma.
\end{equation*}
Since $L^q_\sigma$  is  dual to $L^p_\sigma$ (see \cite{borchers_$l^2$_1988}),  we deduce that
\begin{equation*}
\nl p{\partial_t u^n}\leq C\|u^n\|_{H^1}^2.
\end{equation*}
Therefore $\partial_t u^n$ is bounded in $L^\infty(\R_+;L^p)$. By the Ascoli theorem and the compact embedding $L^p\hookrightarrow W^{-1,p}_{loc}$, we infer, passing to subsequences as necessary,  that:
\begin{equation}\label{unifconv}
u^n\to u \quad\text{strongly in }C^0(\R_+;W^{-1,p}_{loc}).
\end{equation}

Next we have that:
\begin{claim} \label{extraregul}
The operator $\B$ is bounded from $\bm(\hp)$  to $W^{2,p}_{loc}(\hp)$ for all $p<2$.
\end{claim}
\begin{proof}[Proof of Claim:]
From the results of \cite{borchers_$l^2$_1988} we know that, for all $1<r<\infty$, the operator $\varphi\mapsto \nabla^2(\I+\al\A)^{-1}\varphi$ is bounded from $L^r_\sigma$ to $L^r$. By interpolation, we infer that it is also bounded from $L^{2,\infty}_\sigma$ to $L^{2,\infty}$. The embedding $L^{2,\infty}\hookrightarrow L^p_{loc}$ implies that $\varphi\mapsto \nabla^2(\I+\al\A)^{-1}\varphi$ is bounded from $L^{2,\infty}_\sigma$ to $L^p_{loc}$. A similar argument shows that the operators  $\varphi\mapsto \nabla(\I+\al\A)^{-1}\varphi$ and $\varphi\mapsto (\I+\al\A)^{-1}\varphi$ are bounded from $L^{2,\infty}_\sigma$ to $L^p_{loc}$. We conclude that the operator $(\I+\al\A)^{-1}$ is bounded from $L^{2,\infty}_\sigma(\hp)$ to $W^{2,p}_{loc}(\hp)$ for all $p<2$.

The conclusion now follows from Lemma \ref{lem-kh}. Indeed we have that $K_\hp[q]\in L^{2,\infty}_\sigma$ for all $q\in\bm(\hp)$ and $\B(q)=(\I+\al\A)^{-1}K_\hp[q]$.
\end{proof}

Since $q^n$ is bounded in $\bm(\hp)$ we find, in view of Claim \ref{extraregul}, that $u^n$ is bounded in
$L^\infty(\R_+; W^{2,p}_{loc}(\hp))$ for all $p<2$. Then, interpolation together with the uniform convergence \eqref{unifconv} yield
\begin{equation*}
u^n\to u \quad\text{strongly in }C^0(\R_+;W^{s,p}_{loc})\qquad \forall s<2,\, p<2.
\end{equation*}
Using Sobolev embeddings we further deduce that
\begin{equation*}
u^n\to u \quad\text{strongly in }C^0(\R_+\times \hp).
\end{equation*}

Recalling the weak convergence of $q^n$ expressed in \eqref{qnweakconv} we finally deduce that
\begin{equation*}
u^nq^n\to uq\quad\text{in }\mathscr{D}'(\R_+^*\times \hp)
\end{equation*}
so
\begin{equation*}
\dive(u^nq^n)\to \dive(uq)\quad\text{in }\mathscr{D}'(\R_+^*\times \hp).
\end{equation*}

We infer that $q$ is a solution of the $\al$--Euler equations. Moreover, the bound \eqref{qbound}  follows from \eqref{qnbound} and \eqref{qnweakconv}.

It remains to prove the bound \eqref{ubound}. We proceed in the following manner. From \eqref{unifconv} we deduce that $u^n(t)\to u(t)$ in $\mathscr{D}'$ for all $t$. Since $u^n(t)$ is bounded in $H^1$ we infer that $u^n(t)\rightharpoonup u(t)$ weakly in $H^1$ for all $t$. Therefore
\begin{equation*}
\nl2{u(t)}\leq\liminf_{n\to\infty}\nl2{u^n(t)}
\end{equation*}
and
\begin{equation*}
\nl2{\nabla u(t)}\leq\liminf_{n\to\infty}\nl2{\nabla u^n(t)}
\end{equation*}
for all $t\geq0$. We finally deduce from \eqref{convl2} and \eqref{unbound} that
\begin{align*}
\nl2{u(t)}^2+\al\nl2{\nabla u(t)}^2
&\leq\liminf_{n\to\infty}(\nl2{u^n(t)}^2+\al\nl2{\nabla u^n(t)}^2)\\
&=\liminf_{n\to\infty}(\nl2{u^n(0)}^2+\al\nl2{\nabla u^n(0)}^2)\\
&=\nl2{u_0}^2+\al\nl2{\nabla u_0}^2
\end{align*}
for all $t\geq0$. This proves \eqref{ubound} and completes the proof of Theorem \ref{theo-existence}.
\end{proof}

\section{Finding velocity from potential vorticity: interior and boundary parts} \label{sect-intbdry}

In Section \ref{sect-green} we found the solution operator which gives the velocity in terms of potential vorticity. The resulting expression, however, is not explicit enough for us to pass to the limit $\alpha \to 0$. Here we will produce a representation formula for $u$ in terms of $q$ in which we write  $u$ as a sum of a vector field $\uint$, constructed using the method of images but with possibly non-vanishing boundary value, and a vector field $\ubdry$ which corrects the boundary condition.

We will use the notation introduced in Section \ref{sect-notandprelim}, particularly \eqref{Gal}, \eqref{defBS} and \eqref{defkal}.

Let us start by introducing
\begin{equation}\label{def-ubar}
  \uint\equiv  G_\al\ast( K \ast\tq) \equiv K_\al \ast \tq.
\end{equation}

Then
\begin{equation*}
(\mathbb{I}-\al\Delta)\uint=K\ast\tq \quad\text{and}\quad\dive \uint=0\quad\text{in }\R^2.
\end{equation*}
Moreover, since $G_\al$ is radial it follows that $\uint$ inherits the symmetry properties of $K\ast\tq$. In particular, the second component of $\uint$ is odd with respect to $x_2$, so that $\uint^2$ vanishes at $x_2=0$.

Let
\begin{equation}\label{defw}
  \ubdry=u-\uint.
\end{equation}
Then $\ubdry$ is divergence free in $\hp$ and tangent to the boundary of $\hp$. Moreover,
\begin{equation*}
\curl(\I-\al\Delta)\ubdry =\curl(\I-\al\Delta)u-\curl(\I-\al\Delta)\uint=q-\curl(K\ast\tq)=0\quad\text{in }\hp.
\end{equation*}
Thus there exists some scalar function $p$ such that
\begin{equation*}
\ubdry-\al\Delta \ubdry+\nabla p=0 \quad\text{in }\hp.
\end{equation*}

We denote by $g$ the trace of the first component of $\ubdry$, $\ubdry^1$, on the boundary of $\hp$:
\begin{equation}\label{defginit}
g(x_1)\equiv \ubdry^1(x_1,0)=-\uint^1(x_1,0)=-G_\al\ast (K\ast\tq)^1(x_1,0).
\end{equation}

We conclude that $w=\ubdry$ satisfies the following system of equations
\begin{align}
  w-\al\Delta w+\nabla p&=0\label{stokesb1}\quad\text{in }\hp\\
\dive w&=0\label{stokesb2}\quad\text{in }\hp\\
w_1\bigl|_{x_2=0}&=g(x_1)\label{stokesb3}\\
w_2\bigl|_{x_2=0}&=0.\label{stokesb4}
\end{align}

We will find a formula for the solution of this problem through a method which was employed by Solonnikov \cite{solonnikov_estimates_1977} to find the Green's function of the evolutionary Stokes operator in the half-space. Let $\astu $ denote the convolution in the first variable and  $\asth $ denote the convolution on $\hp$ defined as follows. If $\varphi,\psi$ are functions defined on $\hp$, the convolution $\varphi\asth \psi$ is a function defined on $\hp$ whose value is
\begin{equation*}
\varphi \asth \psi(x_1,x_2)=\int_\R\int_0^{x_2} \varphi(y)\psi(x-y)\,dy
\end{equation*}

Next we introduce a pair of functions which will appear frequently in what follows:
\begin{equation}\label{eta1}
  \eta_1=\eta_1(x_1,x_2) = \frac{x_2^2-x_1^2}{|x|^4}, \text{ and}
\end{equation}

\begin{equation}\label{eta2}
  \eta_2=\eta_2(x_1,x_2) = \frac{x_1x_2}{|x|^4}.
\end{equation}

With this notation we have the following result.

\begin{proposition}\label{prop-green}
The solution $w=(w_1,w_2)$ of the problem \eqref{stokesb1}--\eqref{stokesb4} is given by the following formula:
\begin{align*}
  w_1&=-2\al g\astu\partial_2G_\al+\frac{2\al}\pi g\astu\partial_2G_\al\asth\eta_1\\
w_2&=-\frac{4\al}\pi g\astu \partial_2G_\al\asth \eta_2.
\end{align*}
\end{proposition}

\begin{remark}
We note that, if $\varphi$, $\psi$ and $\rho$ are functions on $\hp$, then
\[\varphi \astu (\psi \asth \rho) = (\varphi \astu \psi )\asth \rho \]
as long as $\varphi$ is independent of $x_2$.
\end{remark}

\begin{proof}
We will use the notations $\F_1$ and $\tilde{}$ introduced in \eqref{tildenotation}. Taking the divergence of \eqref{stokesb1} yields $\Delta p=0$. We apply $\F_1$ and deduce that $(\partial_{x_2}^2-\xi_1^2)\wt p=0$. So $\wt p$ is a linear combination of $e^{\pm x_2|\xi_1|}$ with coefficients functions of $\xi_1$. Because $\wt p$ can't exhibit exponential growth at infinity, we infer that
\begin{equation*}
  \wt p(\xi_1,x_2)=C_1(\xi_1)e^{-x_2|\xi_1|}.
\end{equation*}

Let now
\begin{equation*}
  A=w+\nabla p.
\end{equation*}
Because $p$ is harmonic, relation \eqref{stokesb1} implies that $A-\al\Delta A=0$. We apply as above $\F_1$ and deduce that $(\partial_{x_2}^2-\xi_1^2-\frac1\al)\wt A=0$. Since $\wt A$ can't exhibit exponential growth at infinity either, we infer that there exists $C_2(\xi_1)$ and $C_3(\xi_1)$ such that
\begin{equation*}
\wt A(\xi_1,x_2)=e^{-x_2\xial}
\begin{pmatrix}
  C_2(\xi_1)\\ C_3(\xi_1)
\end{pmatrix}.
\end{equation*}

We conclude that
\begin{equation}\label{uhat}
\begin{aligned}
\wt w(\xi_1,x_2)&=\F_1(A-\nabla p)\\
&=e^{-x_2\xial}
\begin{pmatrix}
  C_2(\xi_1)\\ C_3(\xi_1)
\end{pmatrix}
-
\begin{pmatrix}
  i\xi_1\\ \partial_{x_2}
\end{pmatrix}
\bigl[C_1(\xi_1)e^{-x_2|\xi_1|}\bigr]\\
&=e^{-x_2\xial}
\begin{pmatrix}
  C_2(\xi_1)\\ C_3(\xi_1)
\end{pmatrix}
-
e^{-x_2|\xi_1|}C_1(\xi_1)\begin{pmatrix}
  i\xi_1\\ -|\xi_1|
\end{pmatrix}.
\end{aligned}
\end{equation}

Applying the Fourier transform $\F_1$ to \eqref{stokesb2}--\eqref{stokesb4} we get
\begin{align}
i\xi_1\wt w_1+\partial_{x_2}\wt w_2&=0\label{fi1}\\
\wt w_1(\xi_1,0)&=\wh g(\xi_1)\label{fi2}\\
 \wt w_2(\xi_1,0)&=0. \label{fi3}
\end{align}
Using \eqref{uhat} in the three equations above yields a linear system of three equations in the unknowns $C_1$, $C_2$ and $C_3$. This system can be easily solved allowing in return to compute $\wt w$. Plugging \eqref{uhat} in \eqref{fi2} gives
\begin{equation*}
C_2(\xi_1)= \wh g(\xi_1)+i\xi_1 C_1(\xi_1)
\end{equation*}
and in \eqref{fi3} gives
\begin{equation*}
  C_3(\xi_1)=-|\xi_1|C_1(\xi_1)
\end{equation*}
so that
\begin{equation*}
\wt w(\xi_1,x_2)=  e^{-x_2\xial}
       \begin{pmatrix}
         \wh g(\xi_1)\\0
       \end{pmatrix}
+
C_1(\xi_1)\begin{pmatrix}
  i\xi_1\\ -|\xi_1|
\end{pmatrix}(e^{-x_2\xial}-e^{-x_2|\xi_1|}).
\end{equation*}
Using this in \eqref{fi1} yields the following value for $C_1$:
\begin{equation*}
C_1(\xi_1)=\frac{i\xi_1}{\xi_1^2-|\xi_1|\xial}\wh g(\xi_1)=-i\al \frac{\xi_1}{|\xi_1|}(|\xi_1|+\xial)\wh g(\xi_1).
\end{equation*}

So
\begin{align*}
\wt w&=e^{-x_2\xial}
       \begin{pmatrix}
         \wh g(\xi_1)\\0
       \end{pmatrix}
+
 \al \begin{pmatrix}
|\xi_1|\\ i\xi_1
  \end{pmatrix}
(|\xi_1|+\xial)(e^{-x_2\xial}-e^{-x_2|\xi_1|})\wh g(\xi_1).
\end{align*}
Observing that
\begin{equation*}
  \frac{e^{-x_2\xial}-e^{-x_2|\xi_1|}}{|\xi_1|-\xial}
=\int_0^{x_2}e^{-y_2\xial}e^{-|\xi_1|(x_2-y_2)}\,dy_2
\end{equation*}
we finally find the following formula for $\wt w$:
\begin{equation}\label{formulauhat}
\wt w=e^{-x_2\xial}
       \begin{pmatrix}
         \wh g(\xi_1)\\0
       \end{pmatrix}
+
  \begin{pmatrix}
i\frac{\xi_1}{|\xi_1|}\\ -1
  \end{pmatrix}
i\xi_1\wh g(\xi_1)\int_0^{x_2}e^{-y_2\xial}e^{-|\xi_1|(x_2-y_2)}\,dy_2.
\end{equation}

We would now like to take the inverse Fourier transform in the first variable to find the formula for $w$. We deal first with the simplest term above, that is the term $e^{-x_2\xial}\wh g(\xi_1)$. Using \eqref{g1} we have that
\begin{equation}\label{u11}
  \F_1^{-1}[e^{-x_2\xial}\wh g(\xi_1)]=g\astu\F_1^{-1}(e^{-x_2\xial})=-2\al g\astu\partial_2G_\al
\end{equation}
where we recall that $\astu $ denotes the convolution in the first variable.

\medskip

Now let us express the other terms appearing in \eqref{formulauhat}. Taking the second component of \eqref{formulauhat} yields
\begin{equation*}
\wt w_2=-i\xi_1\wh g(\xi_1)\int_0^{x_2}e^{-y_2\xial}e^{-|\xi_1|(x_2-y_2)}\,dy_2.
\end{equation*}
We apply $\F_1^{-1}$ above. Recalling \eqref{four2} and \eqref{g1} we can write
\begin{align*}
 w_2
&=- \F_1^{-1}(\wh g(\xi_1))\astu \int_0^{x_2}\F_1^{-1}\bigl(e^{-y_2\xial}\bigr)\astu \F_1^{-1}\bigl(i\xi_1e^{-|\xi_1|(x_2-y_2)}\bigr)\,dy_2\\
&=-g\astu \int_0^{x_2}(-2\al\partial_2G_\al(\cdot,y_2))\astu \partial_1\F_1^{-1}\bigl(e^{-|\xi_1|(x_2-y_2)}\bigr)\,dy_2\\
&=-g\astu \int_0^{x_2}(-2\al\partial_2G_\al(\cdot,y_2))\astu \partial_1 \rho(\cdot,x_2-y_2)\,dy_2,
\end{align*}
with $\rho(x_1,x_2) \equiv x_2/|x|^2$.

Now, $\partial_1 \rho =-2\eta_2$, with $\eta_2$ as in \eqref{eta2}, so that
\begin{equation}\label{u2}
w_2=\frac{2\al}\pi g\astu \partial_2G_\al\asth \partial_1 \rho
=-\frac{4\al}\pi g\astu \partial_2G_\al\asth \eta_2.
\end{equation}

We can deduce in a similar fashion using \eqref{four4} and \eqref{g1} that
\begin{align*}
  \F_1^{-1}\Bigl[|\xi_1|\wh g(\xi_1)\int_0^{x_2}e^{-y_2\xial}e^{-|\xi_1|(x_2-y_2)}\,dy_2\Bigr]
&=g\astu\int_0^{x_2}\F_1^{-1}\bigl(e^{-y_2\xial}\bigr)\astu\F_1^{-1}\bigl(|\xi_1|e^{-|\xi_1|(x_2-y_2)}\bigr)\,dy_2\\
&=g\astu\int_0^{x_2}(-2\al\partial_2G_\al(\cdot,y_2))\astu \frac{1}{\pi}\eta_1(\cdot,x_2-y_2) \,dy_2\\
&=-\frac{2\al}\pi g\astu\partial_2G_\al\asth\eta_1,
\end{align*}
with $\eta_1$ as in \eqref{eta1}.

Putting together relations \eqref{formulauhat}, \eqref{u11}, \eqref{u2} and the above equalities completes the proof of the proposition.
\end{proof}

To be able to give a complete, explicit, formula for $u$ in terms of $q$ it remains to express the boundary data $g$ defined in \eqref{defginit} in terms of $q$. We proceed in the following manner.

Recall that
\begin{equation*}
g(x_1)=-(G_\al\ast (K\ast\tq)^1)(x_1,0).
\end{equation*}
Recall that $K=\nabla^\perp G$, where $G=\frac1{2\pi}\ln|x|$ is the Green function of the Laplacian in $\R^2$. Therefore
\begin{equation*}
g(x_1) = - ( G_\al \ast (\nabla^\perp G)^1\ast\tq )(x_1,0).
\end{equation*}
Now, since $G_\al$ is the Green function of $\mathbb{I}-\al\Delta$ in $\R^2$ we have that $(\mathbb{I}-\al\Delta)G_\al=\delta$ so $G_\al=\al\Delta G_\al+\delta$. Recall the vector field $H_\al$ introduced in \eqref{defhalpha}. We have, then:
\begin{align*}
G_\al \ast \nabla^\perp G
&=(\al\Delta G_\al+\delta)\ast  \nabla^\perp G\\
&=\al\Delta G_\al\ast\nabla^\perp G+\nabla^\perp G\\
&=\al G_\al\ast\nabla^\perp\Delta G+\nabla^\perp G\\
&=\al G_\al\ast\nabla^\perp\delta+\nabla^\perp G\\
&=\nabla^\perp (\al G_\al+G)\\
&= \nabla^\perp H_\al.
\end{align*}
We infer that
\begin{align*}
g(x_1)=(\partial_2H_\al\ast\tq)(x_1,0)
\end{align*}
Using that $\partial_2H_\al$ and $\tq$ are both odd with respect to $x_2$ we finally get that
\begin{equation*}
g(x_1)=-2\int_\hp  \partial_2H_\al(x_1-y_1,y_2)q(y)\,dy.
\end{equation*}

From \eqref{def-ubar} we have that
\begin{equation}\label{rel-ubar}
\uint=G_\al\ast \frac{x^\perp}{2\pi|x|^2}\ast\tq=K_\al\ast\tq
\end{equation}
where
\begin{equation*}
K_\al=G_\al\ast \frac{x^\perp}{2\pi|x|^2}=G_\al \ast \nabla^\perp G=\nabla^\perp H_\al.
\end{equation*}

We deduce from relations \eqref{defw}, \eqref{rel-ubar} and Proposition \ref{prop-green} the following formula for the solution of \eqref{velpotvort}.
\begin{proposition}\label{prop-formula}
The solution of  \eqref{velpotvort} is given by
\begin{equation*}
u=\uint+\ubdry
\end{equation*}
where $\uint$ is the interior part
\begin{equation*}
\uint\equiv K_\al\ast\tq
\end{equation*}
and $\ubdry$ is the contribution of the boundary which takes the form
\begin{align}
  \ubdry^1&\equiv-2\al g\astu\partial_2G_\al+\frac{2\al}\pi g\astu\partial_2G_\al\asth\eta_1 \label{ubdry1}\\
\ubdry^2&\equiv-\frac{4\al}\pi g\astu \partial_2G_\al\asth \eta_2 \notag
\end{align}
and
\begin{equation}\label{defg}
g(x_1)=-2\int_\hp  \partial_2H_\al(x_1-y_1,y_2)q(y)\,dy.
\end{equation}
\end{proposition}
Above, $\eta_1$ and $\eta_2$ were introduced in \eqref{eta1}, \eqref{eta2}.

\section{Boundary part estimates}\label{sect-l1est}

The purpose of this section is to obtain estimates for the boundary correction term $\ubdry$ in the interior of $\hp$. We will show that, for any $\varepsilon > 0$, if $x_2 > \varepsilon$ then $\ubdry$ is bounded, uniformly with respect to $x_1$ and $\alpha \ll 1$, in terms of the total variation of the potential vorticity.

We begin with a monotonicity result.
\begin{lemma}\label{signhalpha}
We have that $\partial_2 H_\al\geq0$ in the half-plane $\hp$.
\end{lemma}
\begin{proof}
The function $G_1$, which is a Bessel potential, satisfies:
\begin{equation*}
G_1(x)=\frac1{4\pi}\int_0^\infty \frac1t e^{-\frac{\pi|x|^2}t}e^{-\frac t{4\pi}}\,dt,
\end{equation*}
see \cite[page 132]{stein_singular_1970}.

Assume that $x_2>0$. We use  \eqref{galphag1} and \eqref{defhalpha} to deduce that
\begin{align*}
\partial_2H_\al
&=\al\partial_2G_\al+\frac{x_2}{2\pi|x|^2}\\
&=\frac1{\sqrt\al}\partial_2G_1\bigl(\frac x{\sqrt\al}\bigr)+\frac{x_2}{2\pi|x|^2}\\
&=-\frac{x_2}{2\al}\int_0^\infty \frac1{t^2} e^{-\frac{\pi|x|^2}{t\al}}e^{-\frac t{4\pi}}\,dt+\frac{x_2}{2\pi|x|^2}\\
&\geq-\frac{x_2}{2\al}\int_0^\infty \frac1{t^2} e^{-\frac{\pi|x|^2}{t\al}}\,dt+\frac{x_2}{2\pi|x|^2}\\
&=0.
\end{align*}
\end{proof}

Next, we will show that the boundary value of $\ubdry$, $g$, see Proposition \ref{prop-formula}, can be estimated in $L^1$ by the total variation of the potential vorticity.

\begin{proposition}\label{propg}
Assume that $q\in\bm(\hp)$ and let $g$ be the boundary value of $\ubdry$, as defined through \eqref{defg}. Then $g\in L^1(\R)$ and we have
\begin{equation*}
\|g\|_{L^1(\R)}\leq \|q\|_{\bm(\hp)}.
\end{equation*}
\end{proposition}
\begin{proof}
According to Lemma \ref{signhalpha} we know that $\partial_2H_\al(x_1-y_1,y_2)$, the kernel in \eqref{defg}, is non-negative. We take the absolute value in \eqref{defg}, we integrate and we use the Fubini theorem to deduce that
\begin{equation*}
\int_\R |g(x_1)|\,dx_1\leq2\int_\hp  \Bigl(\int_\R\partial_2H_\al(x_1,y_2)\,dx_1\Bigr)\,d|q|(y).
\end{equation*}
We find, for $y_2>0$, that:
\begin{equation*}
\int_\R\partial_2H_\al(x_1,y_2)\,dx_1
=\F_1\partial_2H_\al(0,y_2)
=\frac12-\frac1{2}e^{-\frac{y_2}{\sqrt\al}}\leq \frac12,
\end{equation*}
where we used \eqref{fhalpha}. Therefore
\begin{equation*}
\int_\R |g(x_1)|\,dx_1\leq \int_\hp 1\,d|q|(y)=\|q\|_{\bm(\hp)},
\end{equation*}
as we wished.
\end{proof}

We proceed with two technical lemmas.
\begin{lemma}\label{boundhom}
Let $\eta$ be a function homogeneous of degree $\gamma$ with $\gamma<-\frac12$ and smooth on $\R^2\setminus\{0\}$. There exists a constant $C=C(\eta)$ such that for all $x_2\neq0$ we have that
\begin{equation*}
\|\eta(\cdot,x_2)\|_{L^2(\R)}\leq C x_2^{\gamma+\frac12}.
\end{equation*}
\end{lemma}
\begin{proof}
We have that $|\eta(x)|\leq C|x|^\gamma$ so
\begin{equation*}
\int_\R|\eta(x)|^2\,dx_1\leq C\int_\R|x|^{2\gamma}\,dx_1
=C\int_\R(x_1^2+x_2^2)^\gamma\,dx_1
=Cx_2^{2\gamma+1}\int_\R(t^2+1)^\gamma\,dt
\end{equation*}
where we changed variables $x_1=tx_2$.
\end{proof}
\begin{lemma}\label{lem-d2g}
There exists a universal constant $C$ such that, for all $x_2>0$, we have:
\begin{gather}
\|\partial_2 G_\al (\cdot,x_2)\|_{L^\infty(\R)}\leq \frac C{\al x_2} e^{-\frac{x_2}{\sqrt\al}}\label{d2g1},\\
\|\partial_2 G_\al (\cdot,x_2)\|_{L^2(\R)}\leq \frac C{\al \sqrt{x_2}} e^{-\frac{x_2}{\sqrt\al}}\label{d2g2},\\
\intertext{and}
\|\partial_1\partial_2 G_\al (\cdot,x_2)\|_{L^2(\R)}\leq \frac C{\al x_2\sqrt{x_2}} e^{-\frac{x_2}{\sqrt\al}}.\notag
\end{gather}
\end{lemma}
\begin{proof}
Using \eqref{g1} we obtain
\begin{align*}
\|\partial_2 G_\al (\cdot,x_2)\|_{L^\infty(\R)}
&\leq \frac1{2\pi}\|\F_1\partial_2  G_\al (\cdot,x_2)\|_{L^1(\R)}\\
&=\frac1{4\pi\al}\int_\R e^{-x_2\sqrt{\xi_1^2+\frac1\al}}\,d\xi_1\\
&=\frac1{4\pi\al}\int_{|\xi_1|<\frac1{\sqrt\al}} e^{-x_2\sqrt{\xi_1^2+\frac1\al}}\,d\xi_1+\frac1{4\pi\al}\int_{|\xi_1|>\frac1{\sqrt\al}} e^{-x_2\sqrt{\xi_1^2+\frac1\al}}\,d\xi_1\\
&\leq \frac1{4\pi\al}\int_{|\xi_1|<\frac1{\sqrt\al}} e^{-\frac{x_2}{\sqrt\al}}\,d\xi_1+\frac1{4\pi\al}\int_{|\xi_1|>\frac1{\sqrt\al}} e^{-x_2|\xi_1|}\,d\xi_1\\
&=\frac1{2\pi\al\sqrt\al}e^{-\frac{x_2}{\sqrt\al}}+\frac1{2\pi\al x_2}e^{-\frac{x_2}{\sqrt\al}}\\
&\leq \frac{C}{\al x_2}e^{-\frac{x_2}{\sqrt\al}},
\end{align*}
where we used that the function $s e^{-s}$ is bounded on $\R$. This establishes \eqref{d2g1}.

To prove \eqref{d2g2} we use the Plancherel theorem and relation \eqref{g1} to write
\begin{equation*}
\|\partial_2 G_\al (\cdot,x_2)\|_{L^2(\R)}=\frac1{\sqrt{2\pi}}\|\F_1\partial_2  G_\al (\cdot,x_2)\|_{L^2(\R)}=\frac1{2\al\sqrt{2\pi}}\|e^{-x_2\sqrt{\xi_1^2+\frac1\al}}\|_{L^2(d\xi_1)}.
\end{equation*}
Next we have:
\begin{align*}
\|e^{-x_2\sqrt{\xi_1^2+\frac1\al}}\|^2_{L^2(d\xi_1)}
&=\int_\R e^{-2x_2\sqrt{\xi_1^2+\frac1\al}}\,d\xi_1\\
&=\int_{|\xi_1|<\frac1{\sqrt\al}} e^{-2x_2\sqrt{\xi_1^2+\frac1\al}}\,d\xi_1+\int_{|\xi_1|>\frac1{\sqrt\al}} e^{-2x_2\sqrt{\xi_1^2+\frac1\al}}\,d\xi_1\\
&\leq \int_{|\xi_1|<\frac1{\sqrt\al}} e^{-\frac{2x_2}{\sqrt\al}}\,d\xi_1+\int_{|\xi_1|>\frac1{\sqrt\al}} e^{-2x_2|\xi_1|}\,d\xi_1\\
&= \frac2{\sqrt\al} e^{-\frac{2x_2}{\sqrt\al}}+\frac1{x_2} e^{-\frac{2x_2}{\sqrt\al}}\\
&\leq \frac C{x_2} e^{-\frac{2x_2}{\sqrt\al}}.
\end{align*}
This proves \eqref{d2g2}.

Similarly, the Plancherel theorem and relation \eqref{g1} imply that
\begin{equation*}
\|\partial_1\partial_2 G_\al (\cdot,x_2)\|_{L^2(\R)}=\frac1{\sqrt{2\pi}}\|\F_1\partial_1\partial_2  G_\al (\cdot,x_2)\|_{L^2(\R)}=\frac1{2\al\sqrt{2\pi}}\|\xi_1e^{-x_2\sqrt{\xi_1^2+\frac1\al}}\|_{L^2(d\xi_1)}.
\end{equation*}
Furthermore we have:
\begin{align*}
\|\xi_1e^{-x_2\sqrt{\xi_1^2+\frac1\al}}\|^2_{L^2(d\xi_1)}
&=\int_\R \xi_1^2e^{-2x_2\sqrt{\xi_1^2+\frac1\al}}\,d\xi_1\\
&=\int_{|\xi_1|<\frac1{\sqrt\al}} \xi_1^2 e^{-2x_2\sqrt{\xi_1^2+\frac1\al}}\,d\xi_1+\int_{|\xi_1|>\frac1{\sqrt\al}}  \xi_1^2 e^{-2x_2\sqrt{\xi_1^2+\frac1\al}}\,d\xi_1\\
&\leq \int_{|\xi_1|<\frac1{\sqrt\al}} \xi_1^2  e^{-\frac{2x_2}{\sqrt\al}}\,d\xi_1+\int_{|\xi_1|>\frac1{\sqrt\al}}  \xi_1^2 e^{-2x_2|\xi_1|}\,d\xi_1\\
&= \frac2{3\al\sqrt\al} e^{-\frac{2x_2}{\sqrt\al}}+e^{-\frac{2x_2}{\sqrt\al}}\bigl(\frac1{\al x_2}+\frac1{\sqrt\al x_2^2}+\frac1{2x_2^3}\bigr)\\
&\leq \frac C{x_2^3} e^{-\frac{2x_2}{\sqrt\al}}.
\end{align*}
This completes the proof.
\end{proof}

With these estimates in hand we can now establish the main result in this section.
\begin{proposition}\label{prop-ubdry}
Let $\ubdry$ be given as in Proposition \ref{prop-formula}. There exists a universal constant $C>0$ such that
\begin{equation*}
|\ubdry(x)|\leq C \|q\|_{\bm(\hp)}\al^{\frac14}x_2^{-\frac32}\qquad \forall x\in \hp.
\end{equation*}
\end{proposition}
\begin{proof}
We first estimate the first component of $\ubdry$, given by
\[\ubdry^1 = -2\al g\astu\partial_2G_\al+\frac{2\al}\pi g\astu\partial_2G_\al\asth\eta_1, \]
see \eqref{ubdry1}, where $\eta_1$ was defined in \eqref{eta1}.

We use Young's inequality  and \eqref{d2g1} to bound the first term:
\begin{equation}\label{ubdry3}
\|-2\al g\astu\partial_2 G_\al(\cdot,x_2)\|_{L^\infty(\R)}
\leq 2\al\|g\|_{L^1}\|\partial_2 G_\al(\cdot,x_2)\|_{L^\infty(\R)}\leq C\|g\|_{L^1}\frac{e^{-\frac{x_2}{\sqrt\al}}}{ x_2}
\end{equation}

To bound the second term in the expression for $\ubdry^1$ we note that $\eta_1(x)=\frac{x_2^2-x_1^2}{|x|^4} = \partial_1 (x_1/|x|^2) \equiv \partial_1 \zeta$. We write this term as
\begin{align*}
\frac{2\al}\pi g\astu\partial_2G_\al\asth\eta_1(x)
&=\frac{2\al}\pi\int_0^{x_2} g\astu\partial_2G_\al(\cdot,y_2)\astu\eta_1(\cdot,x_2-y_2)\,dy_2\\
&=\frac{2\al}\pi\int_0^{\frac{x_2}2} g\astu\partial_2G_\al(\cdot,y_2)\astu\eta_1(\cdot,x_2-y_2)\,dy_2\\
&\hskip 3cm +\frac{2\al}\pi\int_{\frac{x_2}2}^{x_2} g\astu\partial_2G_\al(\cdot,y_2)\astu\eta_1(\cdot,x_2-y_2)\,dy_2\\
&\equiv I_1+I_2.
\end{align*}

We use Lemmas \ref{boundhom} and \ref{lem-d2g} and the Young inequality to bound $I_1$ as follows:
\begin{align*}
\|I_1(\cdot,x_2)\|_{L^\infty(\R)}
&\leq C\al \int_0^{\frac{x_2}2} \|g\|_{L^1}\|\partial_2G_\al(\cdot,y_2)\|_{L^2(\real)}\|\eta_1(\cdot,x_2-y_2)\|_{L^2(\real)}\,dy_2\\
&\leq C\|g\|_{L^1}\int_0^{\frac{x_2}2}\frac{e^{-\frac{y_2}{\sqrt\al}}}{\sqrt y_2}(x_2-y_2)^{-\frac32}\,dy_2\\
&\leq C\|g\|_{L^1}x_2^{-\frac32}\int_0^{\frac{x_2}2}\frac{e^{-\frac{y_2}{\sqrt\al}}}{\sqrt y_2}\,dy_2\\
&= C \|g\|_{L^1}x_2^{-\frac32}\al^{\frac14}\int_0^{\frac{x_2}{2\sqrt\al}}\frac{e^{-t}}{\sqrt t}\,dt\\
&\leq C \|g\|_{L^1}x_2^{-\frac32}\al^{\frac14}.
\end{align*}

To estimate $I_2$ we write first
\begin{align*}
I_2&=\frac{2\al}\pi\int_{\frac{x_2}2}^{x_2} g\astu\partial_2G_\al(\cdot,y_2)\astu\eta_1(\cdot,x_2-y_2)\,dy_2\\
&=\frac{2\al}\pi\int_{\frac{x_2}2}^{x_2} g\astu\partial_2G_\al(\cdot,y_2)\astu\partial_1\zeta(\cdot,x_2-y_2)\,dy_2\\
&=\frac{2\al}\pi\int_{\frac{x_2}2}^{x_2} g\astu\partial_1\partial_2G_\al(\cdot,y_2)\astu\zeta(\cdot,x_2-y_2)\,dy_2.
\end{align*}
We can now bound as above with the help of  Lemmas \ref{boundhom} and \ref{lem-d2g} and the Young inequality:
\begin{align*}
\|I_2(\cdot,x_2)\|_{L^\infty(\R)}
&\leq C\al \int_{\frac{x_2}2}^{x_2} \|g\|_{L^1}\|\partial_1\partial_2G_\al(\cdot,y_2)\|_{L^2(\real)}\|\zeta(\cdot,x_2-y_2)\|_{L^2(\real)}\,dy_2\\
&\leq C\|g\|_{L^1}\int_{\frac{x_2}2}^{x_2}\frac{e^{-\frac{y_2}{\sqrt\al}}}{y_2\sqrt y_2}\frac1{\sqrt{x_2-y_2}}\,dy_2\\
&\leq C \|g\|_{L^1}x_2^{-\frac32}e^{-\frac{x_2}{2\sqrt\al}}\int_{\frac{x_2}2}^{x_2}\frac1{\sqrt{x_2-y_2}}\,dy_2\\
&\leq C \|g\|_{L^1}\frac{e^{-\frac{x_2}{2\sqrt\al}}}{x_2}.
\end{align*}

The estimates for $I_1$ and $I_2$ obtained above yield the following bound:
\begin{equation}\label{ubdry4}
|\frac{2\al}\pi g\astu\partial_2G_\al\asth\frac{x_2^2-x_1^2}{|x|^4}|
\leq C \|g\|_{L^1}\frac{\al^{\frac14}}{x_2^{\frac32}}+C \|g\|_{L^1}\frac{e^{-\frac{x_2}{2\sqrt\al}}}{x_2}.
\end{equation}

Now, by Proposition \ref{propg} we have that $\|g\|_{L^1}$ is bounded by $\|q\|_{\bm(\hp)}$. In addition, the estimate $e^{-\frac z2}\leq \frac 1{\sqrt z}$ applied for $z=\frac{x_2}{\sqrt\al}$ shows that the second term on the rhs above is bounded by the first term. Therefore, \eqref{ubdry3}, \eqref{ubdry4} yield the desired estimate for $\ubdry^1$.

Next recall that the second component of $\ubdry$ is given by
\[\ubdry^2=-\frac{4\al}\pi g\astu \partial_2G_\al\asth \eta_2,\]
where $\eta_2$ is as in \eqref{eta2}.

Noticing that $\eta_2 = \frac{x_1x_2}{|x|^4}=-\partial_1\bigl(\frac{x_2}{2|x|^2}\bigr)$  we can use the same analysis leading up to \eqref{ubdry4} to estimate $\ubdry^2$.

This concludes the proof.
\end{proof}

\section{Passing to the limit $\al\to0$}
\label{sect-passlim}

We are now ready to prove the main result of this work, the convergence, passing to subsequences as needed, of solutions of the $\al$-Euler equations in the half-plane, with no-slip boundary conditions, to solutions of the Euler equations, assuming a sign condition on the singular part of the initial (potential) vorticity. In other words, we show that, if $v$ is any weak limit, in a sense to be made precise, of solutions $u_\al$ of the $\al$-Euler equations in the half-plane, with an initial potential vorticity $q_0 \in \bigl(\bm_+(\hp)+L^1(\hp)\bigr) \cap \dot H^{-1}(\hp)$, such that $u_\al$ vanishes on $\partial \hp$, then $v$ is a weak solution of the incompressible Euler equations in $\hp$, with initial vorticity $q_0$. The proof is based on the proof of the well-known Delort Theorem, see \cite{delort_existence_1991-1}, on the existence of vortex sheet evolution for the 2D Euler equations when vorticity has a distinguished sign. From a technical point-of-view we make use of the idea of symmetrization of the kernels involved in recovering velocity from (potential) vorticity, as was done in \cite{schochet_weak_1995}. Below, the gradient is taken in the spatial variable and not in time.

Let us begin by describing a weak formulation of the $\al$-Euler equations which will be useful for our purposes. Fix $T>0$. Let $q_0\in \bm(\hp)$. Let $u_\al \in C^0_b(\R_+;H^1_w(\hp))\cap C^0(\R_+\times \hp)$ be the solution with initial potential vorticity $q_0$ obtained in Theorem \ref{theo-existence}. Recall that $q_\al \in L^\infty(\R_+;\bm(\hp))$.

Let $\varphi\in C^\infty_c([0,T)\times\hp)$ be a test function. Then $\varphi\in C^\infty_c([0,T)\times\hp^\ep)$ for some $\ep>0$, where $\hp^\ep$ denotes the half-plane $\{x_2>\ep\}$. We multiply the vorticity formulation of the $\al$-Euler equation by $\varphi$ and integrate by parts to obtain
\begin{align*}
\displaystyle{\int_0^T\int_\hp }\partial_t \varphi\,dq_\al dt+ \int_\hp \varphi(0,x)dq_0(x)
  & =  -\int_0^T\int_\hp u_\al\cdot\nabla\varphi\,dq_\al dt \\
&=-\int_0^T\int_\hp \ualint\cdot\nabla\varphi\,dq_\al dt 
- \int_0^T\int_\hp \ualbdry\cdot\nabla\varphi\,dq_\al dt
\end{align*}
where we used the decomposition $\ual=\ualint+\ualbdry$ introduced in Proposition \ref{prop-formula}. From Proposition \ref{prop-ubdry} we get that $\|\ualbdry\|_{L^\infty(\hp^\ep)}\leq C_\ep\|\qual\|_{\bm(\hp)}$, for some $C_\ep>0$ which is bounded with respect to $\al$, for $0<\al<1$. From the bound \eqref{qbound} on the $\bm(\hp)$-norm of $\qual$ we infer that
\begin{equation}\label{f2}
\Bigl|\int_0^T\int_\hp \ualbdry\cdot\nabla\varphi\,dq_\al dt\bigr|\leq C_{\ep}T\|q_0\|_{\bm(\hp)}\nl\infty{\nabla\varphi}.
\end{equation}

We write next
\begin{equation*}
\begin{aligned}
\int_0^T\int_\hp \ualint\cdot\nabla\varphi\,dq_\al dt
&=\int_0^T\int_\hp K_\al\ast\tq_\al(x)\cdot\nabla\varphi(t,x)\,dq_\al(x) dt\\
&=\int_0^T\iint_{\hp\times \R^2} K_\al(x-y)\cdot\nabla\varphi(t,x)\,dq_\al(x)\,d\tq_\al(y) dt\\
&=\int_0^T\iint_{\hp\times \hp} [K_\al(x-y)-K_\al(x-\overline{y})]\cdot\nabla\varphi(t,x)\,dq_\al(x)\,d\qual(y) dt\\
&\equiv \int_0^T\iint_{\hp\times \hp} N_\al(t,x,y)\,dq_\al(x)\,d\qual(y) dt,
\end{aligned}
\end{equation*}
where $\overline y=(y_1,-y_2)$ is the image of $y$; above we have used that $\tq_\al$ is the odd extension of $\qual$ to $\R^2$. Next we symmetrize and write
\begin{align*}
\int_0^T\int_\hp \ualint\cdot\nabla\varphi\,dq_\al dt
&=\int_0^T\iint_{\hp\times \hp} \frac{N_\al(t,x,y)+N_\al(t,y,x)}2\,dq_\al(x)\,d\qual(y) dt\\
&\equiv \int_0^T\iint_{\hp\times \hp} H_\varphi^\al(t,x,y)\,dq_\al(x)\,d\qual(y) dt.
\end{align*}
Since $K_\al$ is odd we have that
\begin{multline}\label{defsal}
H_\varphi^\al(x,y)=\frac12 K_\al(x-y)\cdot[\nabla\varphi(t,x)-\nabla\varphi(t,y)]-\frac12 K_\al(x-\overline{y})\cdot\nabla\varphi(t,x)\\ -\frac12 K_\al(y-\overline{x})\cdot\nabla\varphi(t,y).
\end{multline}

We have established that $\ual$ satisfies the following weak formulation:
\begin{multline} \label{weakalE}
  \int_0^T\int_\hp \partial_t \varphi d\qual dt + \int_0^T\iint_{\hp\times \hp} H_\varphi^\al(t,x,y)\,dq_\al(x)\,d\qual(y) dt  \\
  + \int_0^T\int_\hp \ualbdry\cdot\nabla\varphi\,dq_\al dt + \int_\hp \varphi(0,x)dq_0(x)  = 0.
\end{multline}

We now recall the weak vorticity formulation of the incompressible Euler equations in the half-plane for vortex sheet regularity. For a further discussion of weak vorticity formulations in other domains see \cite{lopes_filho_existence_2001} and \cite{iftimie_weak_2019}.

\begin{definition}\label{weakE}
  Let $\omega_0 \in \bm (\hp) \cap \dot H^{-1} (\hp)$. We say $\omega \in L^\infty(0,T;\bm(\hp)\cap \dot H^{-1}(\hp))$ is a solution of the weak vorticity formulation of the incompressible Euler equations, with initial data $\omega_0$, if, for any $\varphi \in C^\infty_c([0,T) \times \hp)$, the following identity holds true:
  \begin{equation*}
\int_0^T\int_\hp \partial_t \varphi d\omega dt + \int_0^T\iint_{\hp\times \hp} H_\varphi(t,x,y)\,d\omega(x)\,d\omega(y) dt
  +  \int_\hp \varphi(0,x)\,d\omega_0(x) dt =0.
      \end{equation*}
 The auxiliary test function $H_\varphi$ is given by
 \begin{equation*}
   H_\varphi=H_\varphi(t,x,y)\equiv \frac{K_\hp(x,y)\cdot \nabla \varphi(t,x)+K_\hp(y,x)\cdot \nabla \varphi(t,y)}{2}.
 \end{equation*}
\end{definition}

In view of the expression for $K_\hp$, see relation \eqref{khalfplane}, it follows that
\begin{multline}\label{defs}
H_\varphi(t,x,y)=\frac12 K(x-y)\cdot[\nabla\varphi(t,x)-\nabla\varphi(t,y)]-\frac12 K(x-\overline{y})\cdot\nabla\varphi(t,x) \\ -\frac12 K(y-\overline{x})\cdot\nabla\varphi(t,y).
\end{multline}

We will need some properties of the kernels $K$ and $K_\al$, which we collect in the following proposition.
\begin{proposition}\label{prop-kal}
There exists a universal constant $C>0$ such that, for all $x\in\R^2\setminus\{0\}$, we have that
\begin{equation*}
|K_\al(x)|\leq\frac C{|x|}\quad\text{and}\quad |K_\al(x)-K(x)|\leq C\frac{\sqrt\al}{|x|^2}.
\end{equation*}
In particular, for any $\theta>0$ we have that $K_\al\alto K$ uniformly in the set $|x|>\theta$.
\end{proposition}
\begin{proof}
The first bound was proved in \cite[pages 703 and 715]{bardos_global_2010}. To prove the second bound we recall that
\begin{equation*}
\F(K)=\F\bigl(\frac{x^\perp}{2\pi|x|^2}\bigr)=-i\frac{\xi^\perp}{|\xi|^2}.
\end{equation*}
Then
\begin{equation*}
\F(K_\al)=\F(K\ast G_\al)=\F(K)\F(G_\al)=-i\frac{\xi^\perp}{|\xi|^2(1+\al|\xi|^2)}.
\end{equation*}

We can now estimate
\begin{multline*}
\||x|^2(K_\al-K)\|_{L^\infty(\R^2)}
\leq \frac1{2\pi}\|\F\bigl[|x|^2(K_\al-K)\bigr]\|_{L^1(\R^2)}
=\frac1{2\pi}\|\Delta\F(K_\al-K)\|_{L^1(\R^2)}\\
=\frac1{2\pi}\left\|\Delta\bigl(\frac{\al\xi^\perp}{1+\al|\xi|^2}\bigr)\right\|_{L^1(\R^2)}
=\frac{\sqrt\al}{2\pi}\left\|\Delta\bigl(\frac{\xi^\perp}{1+|\xi|^2}\bigr)\right\|_{L^1(\R^2)}.
\end{multline*}
This completes the proof of the proposition.
\end{proof}

We will also need  some special properties of the functions $H_\varphi^\al$ and $H_\varphi$. We will see that, in the analysis of convergence of the nonlinear terms in the proof of Theorem \ref{theo-convergence}, time will be treated as a parameter. Thus we omit the dependence on $t$ in both $H_\varphi^\al$ and $H_\varphi$.

\begin{lemma}\label{lem-S}
We have that
\begin{enumerate}
\item The function $H_\varphi$ is smooth on $\overline \hp\times \overline \hp\setminus\supp(\varphi)\times\supp(\varphi)$, supported in $\overline {\hp^\ep}\times \overline \hp\cup \overline \hp\times \overline {\hp^\ep}$ and vanishes on the boundary of $\overline \hp\times \overline \hp$ and at infinity. In addition,  $H_\varphi^\al$ is also supported in $\overline {\hp^\ep}\times \overline \hp\cup \overline \hp\times \overline {\hp^\ep}$.
\item The functions $H_\varphi^\al$ and $H_\varphi$ are uniformly bounded. More precisely, there exists a constant $C_\ep$ such that
\begin{equation}\label{boundS}
|H_\varphi^\al(x,y)|\leq C_\ep \|\varphi\|_{W^{2,\infty}}\quad\text{and}\quad
|H_\varphi(x,y)|\leq C_\ep \|\varphi\|_{W^{2,\infty}}.
\end{equation}
\item For all $\theta>0$ we have that $H_\varphi^\al \alto H_\varphi$ uniformly in the set $|x-y|>\theta$.
\item There exists a non-negative function $F$ continuous on $\overline \hp\times \overline \hp$, vanishing at infinity and such that $|H_\varphi^\al|\leq F$ for all $\al$ and $|H_\varphi|\leq F$.
\end{enumerate}
\end{lemma}
\begin{proof}
We prove first a). The function
\begin{equation*}
N(x,y)=[K(x-y)-K(x-\overline{y})]\cdot\nabla\varphi(x)
\end{equation*}
is clearly smooth on $\overline \hp\times \overline \hp\setminus\supp(\varphi)\times\supp(\varphi)$ and supported in $\overline {\hp^\ep}\times \overline \hp$ (recall that $\varphi$ is supported in $\hp^\ep$). In particular it vanishes if $x\in\partial \hp$. It also vanishes if $y\in\partial \hp$ because for such a $y$ we have that $y=\overline y$. Obviously
\begin{equation}\label{nal}
|N(x,y)|\leq C\bigl(\frac1{|x-y|}+\frac1{|x-\overline y|}\bigr)|\nabla\varphi(x)|\leq C\frac{|\nabla\varphi(x)|}{|x-y|}.
\end{equation}
We infer that $N$ vanishes at infinity. Since
\begin{equation}\label{sal}
H_\varphi(x,y)=\frac{N(x,y)+N(y,x)}2
\end{equation}
we deduce that $H_\varphi$ have all the properties listed in a). We observe in a similar manner that $H_\varphi^\al$ is  supported in $\overline {\hp^\ep}\times \overline \hp\cup \overline \hp\times \overline {\hp^\ep}$. This completes the proof of part a).

To prove part b), we recall the definition of $H_\varphi^\al$ given in \eqref{defsal} and use Proposition \ref{prop-kal} to bound
\begin{align*}
 |H_\varphi^\al(x,y)|
&\leq C\frac{|\nabla\varphi(x)-\nabla\varphi(y)|}{|x-y|}+C\frac{|\nabla\varphi(x)|}{|x-\overline y|}+C\frac{|\nabla\varphi(y)|}{|y-\overline x|}\\
&\leq C\nl\infty{\nabla^2\varphi}+C\frac{|\nabla\varphi(x)|+|\nabla\varphi(y)|}{x_2+y_2}\\
&\leq C\nl\infty{\nabla^2\varphi}+\frac C\ep \nl\infty{\nabla\varphi}\\
& \leq C_\ep \|\varphi\|_{W^{2,\infty}}
\end{align*}
where we also used that $\supp\varphi\subset \hp^\ep$. The same argument works for $H_\varphi$ so this proves b).

To prove c), we subtract \eqref{defs} from \eqref{defsal} and use the last estimate from Proposition \ref{prop-kal} to bound
\begin{align*}
|H_\varphi^\al(x,y)-H_\varphi(x,y)|
&\leq C\sqrt\al\frac{|\nabla\varphi(x)-\nabla\varphi(y)|}{|x-y|^2}+C\sqrt\al\frac{|\nabla\varphi(x)|}{|x-\overline y|^2}+C\sqrt\al\frac{|\nabla\varphi(y)|}{|y-\overline x|^2}\\
&\leq C\sqrt\al\frac{|\nabla\varphi(x)-\nabla\varphi(y)|}{|x-y|^2}+C\sqrt\al\frac{|\nabla\varphi(x)|}{\ep^2}+C\sqrt\al\frac{|\nabla\varphi(y)|}{\ep^2}\\
&\leq C\sqrt\al\frac{\nl\infty{\nabla^2\varphi}}{|x-y|}+C\sqrt\al\frac{\nl\infty{\nabla\varphi}}{\ep^2}.
\end{align*}
If we assume that $|x-y|>\theta>0$ then
\begin{equation*}
|H_\varphi^\al(x,y)-H_\varphi(x,y)|
\leq C\sqrt\al\Bigl(\frac{\nl\infty{\nabla^2\varphi}}{\theta}+\frac{\nl\infty{\nabla\varphi}}{\ep^2}\Bigr)\alto0
\end{equation*}
which shows part c).

It remains to prove d). We use \eqref{nal} and \eqref{sal} to bound
\begin{equation*}
|H_\varphi^\al(x,y)|\leq C\frac{|\nabla\varphi(x)|+|\nabla\varphi(y)|}{|x-y|}
\leq C\frac{|\nabla\varphi(x)|+|\nabla\varphi(y)|+1}{|x-y|}.
\end{equation*}
Recalling the uniform bound \eqref{boundS} one can easily check that the function
\begin{equation*}
F(x,y)=\min\Bigl(C_\ep \|\varphi\|_{W^{2,\infty}},C\frac{|\nabla\varphi(x)|+|\nabla\varphi(y)|+1}{|x-y|} \Bigr)
\end{equation*}
has all the required properties. This completes the proof of the lemma.
\end{proof}

\begin{remark}
  The properties of $H_\varphi$ above have been discussed and used in \cite{lopes_filho_existence_2001} and, in fact, they hold for more general domains. See \cite{iftimie_weak_2019} for a thorough account.
\end{remark}

We are now ready to establish our main result.

\begin{proof}[Proof of Theorem \ref{theo-convergence}]

Let us fix $q_0 \in \bigl(\bm_+(\hp)+L^1(\hp)\bigr)\cap \dot H^{-1}(\hp)$, independent of $\al$. Let $u_\al$, $q_\al$ solve the $\al$-Euler equations as given by Theorem \ref{theo-existence}. Since $q_0\in \dot{H}^{-1}(\hp)$ there exists $f_0\in L^2$ such that $\curl f_0=q_0$. We know that $u_{\al,0}=(\I+\al\A)^{-1}\PP f_0$ so, by Lemma \ref{lem-stokes}, we have that
\begin{equation*}
\nl2{u_{\al,0}}^2+\al\nl2{\nabla u_{\al,0}}^2\leq \nl2{f_0}^2.
\end{equation*}

The energy inequality \eqref{ubound} now implies that
\begin{equation*}
\nl2{u_\al(t)}^2+\al\nl2{\nabla u_\al(t)}^2\leq \nl2{f_0}^2\quad\forall t\geq0
\end{equation*}
so $u_\al$ is bounded in $L^\infty(\R_+;L^2_\sigma)$ uniformly in $\al$. Then there exists some $v\in L^\infty(\R_+;L^2_\sigma)$ and some subsequence of $u_\al$, which we do not relabel, such that
\begin{equation}\label{ualconv}
u_\al\rightharpoonup v\quad\text{in }L^\infty(\R_+;L^2_\sigma)\text{ weak}\ast
\end{equation}
as $\al\to0$.

The bound \eqref{qbound} implies that $q_\al$ is bounded in $L^\infty(\R_+;\bm(\hp))$ uniformly in $\al$. So there exists some $\om\in L^\infty(\R_+;\bm(\hp))$ and some subsequence of $q_\al$, which again we do not relabel, such that
\begin{equation}\label{qualconv}
q_\al\rightharpoonup \om\quad\text{in }L^\infty(\R_+;\bm(\hp))\text{ weak}\ast
\end{equation}
as $\al\to0$.

Taking the curl in \eqref{ualconv} implies that
\begin{equation*}
\curl u_\al\to\curl v\quad\text{in }\mathscr{D}'(\R_+^*\times \hp)
\end{equation*}
as $\al\to0$.
Moreover
\begin{equation*}
\curl\Delta u_\al\to\curl \Delta v\quad\text{in }\mathscr{D}'(\R_+^*\times \hp)
\end{equation*}
and, therefore,
\begin{equation*}
\al \curl\Delta u_\al\to0\quad\text{in }\mathscr{D}'(\R_+^*\times \hp)
\end{equation*}
as $\al\to0$. We infer that
\begin{equation*} 
q_\al=\curl u_\al - \al \curl\Delta u_\al\to\curl v\quad\text{in }\mathscr{D}'(\R_+^*\times \hp)
\end{equation*}
as $\al\to0$. Comparing this with \eqref{qualconv} and using the uniqueness of limits in the sense of the distributions we infer that $\om=\curl v$. Recalling that $v\in L^\infty(\R_+;L^2_\sigma)$ we further deduce that
\begin{equation}\label{omh-1}
\om\in L^\infty(\R_+;H^{-1}).
\end{equation}

It follows that
\begin{equation}\label{linpart}
  \int_0^T\int_\hp \partial_t \varphi d\qual dt \to \int_0^T\int_\hp \partial_t \varphi d\omega dt
\end{equation}
as $\al \to 0$. This is the only linear term in \eqref{weakalE} which we need to analyze, given that $\omega_0\equiv q_0$.

Putting together relations \eqref{f2}, \eqref{weakalE}  and \eqref{boundS} we deduce
\begin{equation*}
|\langle\partial_t\qual,\varphi\rangle_{\mathscr{D}',\mathscr{D}}|
\leq C_\ep\|\varphi\|_{W^{2,\infty}}\leq C_\ep\|\varphi\|_{H^4}.
\end{equation*}
This implies that
\begin{equation*}
\|\partial_t\qual\|_{L^\infty(0,T;H^{-4}(\hp^\ep))}\leq C_\ep,
\end{equation*}
so $\partial_t\qual$ is bounded in $L^\infty(0,T;H^{-4}_{loc})$. By the Ascoli theorem, we find, passing to  subsequences as needed, that
\begin{equation*}
\qual\to\om\quad\text{in }C^0(\R_+;H^{-5}_{loc}).
\end{equation*}

In particular, for all $t\geq0$, we have that $\qual(t)\to\om(t)$ in $H^{-5}_{loc}$. Given that $\qual(t)$ is bounded in $\bm(\hp)$, we deduce that  $\qual(t)\rightharpoonup \om(t)$ weak-$\ast$ $\bm(\hp)$.

Let us now address the nonlinear terms in \eqref{weakalE}.

We apply first  Proposition \ref{prop-ubdry} and use that $\supp\varphi\subset [0,T) \times \hp^\ep$ to bound
\begin{equation}\label{ualbdrylim}
\begin{aligned}
\bigl|\int_0^T\int_\hp \ualbdry\cdot\nabla\varphi\,dq_\al dt\bigr|
&\leq CT\nl\infty{\nabla\varphi}\|\ualbdry\|_{L^\infty((0,T)\times\hp^\ep)}\|\qual\|_{L^\infty(0,T;\bm(\hp))}\\
&\leq CT\nl\infty{\nabla\varphi}\|q_0\|_{\bm(\hp)}^2\al^{\frac14}\ep^{-\frac32}\\
&\alto0.
\end{aligned}
\end{equation}
Therefore the nonlinear term in the second line of \eqref{weakalE} vanishes as $\al\to0$.

We now pass to the limit in the nonlinear term in the first line of \eqref{weakalE}, namely
\begin{equation*}
\int_0^T\iint_{\hp\times \hp} H_\varphi^{\al}(t,x,y)\,dq_\al(x)\,d\qual(y)dt.
\end{equation*}
From Proposition \ref{prop-formula} we have that $H_\varphi^\al$ is bounded uniformly in $t$ and $\al$. Since $\qual$ is bounded with respect to $\al$ in $L^\infty(0,T;\bm(\hp))$ it follows that
\[\left|\iint_{\hp\times \hp} H_\varphi^{\al}(t,x,y)\,dq_\al(x)\,d\qual(y)\right| \leq \|H_\varphi^\al\|_{L^\infty}\|q_0\|_{\bm(\hp)}^2.\]
Hence, by the Lebesgue Dominated Convergence theorem, it suffices to pass to the limit for a fixed time $t$, hereafter omitted. From \eqref{omh-1} we know that that $\om(t)\in H^{-1}(\hp)$ a.e. in time, so we can assume in the sequel that $\om\in H^{-1}(\hp)$.

We will prove that
\begin{equation}\label{convfinal}
\iint_{\hp\times \hp} H_\varphi^{\al}(t,x,y)\,dq_\al(x)\,d\qual(y)
\alto \iint_{\hp\times \hp} H_\varphi(t,x,y)\,d\om(x)\,d\om(y)
\end{equation}

We know that  $\qual(t)\rightharpoonup \om(t)$ weak-$\ast$ $\bm(\hp)$. Since $|\qual(t)|$ is bounded in $\bm(\hp)$, it admits a sub-sequence, relabeled  $|\qual(t)|$, which converges weak-$\ast$ $\bm(\hp)$ to some measure $\mu\in\bm_+(\hp)$. This sub-sequence depends on the time $t$ and we cannot assume that we can choose the same sub-sequence for all times $t$. But since the limit we seek to find in \eqref{convfinal} does not depend on $\mu$, proving the convergence on this time-dependent sub-sequence implies the convergence for the whole sequence without extracting the time-dependent sub-sequence.

A crucial property of the measures $\om$ and $\mu$ is that they are continuous: $\om(\{x\})=\mu(\{x\})=0$ for all $x\in \hp$.
\begin{claim}\label{contmes}
The measures $\om$, $|\om|$ and $\mu$ are continuous on $\hp$.
\end{claim}
\begin{proof}[Proof of Claim]
We know that $\om\in H^{-1}(\hp)$ so it is a continuous measure. It follows, therefore, that $|\om|$ is also a continuous measure.

Next let us prove that $\mu$ is a continuous measure. Since $\qual$ verifies the transport equation \eqref{potvorteq}, we know from the DiPerna-Lions theory (see \cite[Theorem III.2]{diperna_ordinary_1989}) that $\qual(t)$ is  the image measure of $q_0$ by a measure preserving flow map $X(t)$: $\qual(t)=q_0\circ X(t)$.

By hypothesis $q_0\in \bm_+(\hp)+L^1(\hp)$ so there exists $q_1\in\bm_+(\hp)$ and $q_2\in L^1(\hp)$ such that $q_0=q_1+q_2$. Let $\qual^1=q_1\circ X(t)$ and  $\qual^2=q_2\circ X(t)$ so that $\qual=\qual^1+\qual^2$. Since the flow map $X(t)$ is volume preserving, the measures $\qual^1$ and $\qual^2$ are bounded in $\bm(\hp)$. Extracting sub-sequences if necessary, we can assume without loss of generality that
\begin{gather*}
\qual^1\weakalto\om_1 \quad \text{weak$\ast$ in }\bm(\hp)\\
\qual^2\weakalto\om_2 \quad \text{weak$\ast$ in }\bm(\hp)\\
\intertext{and}
|\qual^2|\weakalto\om_3 \quad \text{weak$\ast$ in }\bm(\hp).
\end{gather*}
Since $q_2\in L^1(\hp)$ and $X(t)$ is volume preserving, we observe that the sequence  $\qual^2=q_2\circ X(t)$ is equi-integrable. The Dunford-Pettis theorem implies that $\om_2\in L^1(\hp)$. Similarly  $\om_3\in L^1(\hp)$. Recall also that $q_1\in\bm_+(\hp)$ so $\qual^1\in\bm_+(\hp)$ which in turn implies that $\om_1\in\bm_+(\hp)$.

Since $\qual$ converges to $\om$ weak$\ast$ in $\bm(\hp)$ we have that $\om=\om_1+\om_2$. We know that $\om$ is a continuous measure. But $\om_2$ is also continuous as a measure because it is an $L^1$ function. We conclude that $\om_1=\om-\om_2$ is a continuous measure too.

Next we bound
\begin{equation*}
|\qual|=|\qual^1+\qual^2|\leq |\qual^1|+|\qual^2|=\qual^1+|\qual^2|.
\end{equation*}
This inequality is preserved by the weak$\ast$ limit in $\bm(\hp)$. At the limit we get that
\begin{equation*}
\mu\leq \om_1+\om_3.
\end{equation*}
Since $\om_3$ is an $L^1$ function, it defines a continuous measure. Therefore the rhs above is a continuous measure. Recalling that $\mu\geq0$ we finally deduce that $\mu$ must be a continuous measure. This completes the proof of the claim.
\end{proof}

We  continue now with the proof of the convergence stated in \eqref{convfinal}.

Let $0<\theta<\ep/4$ and consider a function $\chi\in C^\infty(\R^2\times\R^2;[0,1])$ such that $\chi_\theta(x,y)=1$ if $|x-y|<\theta/2$ and $\chi_\theta(x,y)=0$ if $|x-y|>\theta$. We decompose
\begin{equation}\label{I1234}
 \iint_{\hp\times \hp} H_\varphi^{\al}\,dq_\al(x)\,d\qual(y)-\iint_{\hp\times \hp} H_\varphi\,d\om(x)\,d\om(y)
= I_1+I_2+I_3-I_4
\end{equation}
where
\begin{align*}
I_1&=\iint_{\hp\times \hp} H_\varphi^{\al}\chi_\theta\,dq_\al(x)\,d\qual(y)\\
I_2&=\iint_{\hp\times \hp} (H_\varphi^{\al}-H_\varphi)(1-\chi_\theta)\,dq_\al(x)\,d\qual(y)\\
I_3&=\iint_{\hp\times \hp} H_\varphi(1-\chi_\theta)\,dq_\al(x)\,d\qual(y)-\iint_{\hp\times \hp} H_\varphi(1-\chi_\theta)\,d\om(x)\,d\om(y) \\
\intertext{and}
I_4&=\iint_{\hp\times \hp} H_\varphi\chi_\theta\,d\om(x)\,d\om(y)
\end{align*}

From Lemma  \ref{lem-S} c) we know that $(H_\varphi^{\al}-H_\varphi)(1-\chi_\theta)\to0$ uniformly as $\al\to0$  so
\begin{equation}\label{I2}
I_2=\iint_{\hp\times \hp} (H_\varphi^{\al}-H_\varphi)(1-\chi_\theta)\,dq_\al(x)\,d\qual(y)\alto0.
\end{equation}

From Lemma  \ref{lem-S} a) we know that $H_\varphi(1-\chi_\theta)$ is continuous on $\overline \hp\times\overline \hp$ and vanishes at the boundary and at the infinity. Since $\qual\otimes\qual\to\om\otimes\om$ in $\bm(\hp\times \hp)$ we deduce that
\begin{equation}\label{I3}
I_3=\iint_{\hp\times \hp} H_\varphi(1-\chi_\theta)\,dq_\al(x)\,d\qual(y)-\iint_{\hp\times \hp} H_\varphi(1-\chi_\theta)\,d\om(x)\,d\om(y) \alto0.
\end{equation}

Next, let $\eta_\ep\in C^0(\overline \hp;[0,1])$ be such that $\eta_\ep(x)=1$ if $x_2>\ep$ and $\eta_\ep(x)=0$ if $x_2<\ep/2$. The function $\eta_\ep(x)+\eta_\ep(y)$ is greater than 1 on $\overline {\hp^\ep}\times \overline \hp\cup \overline \hp\times \overline {\hp^\ep}$. We know from  Lemma  \ref{lem-S} a) that $H_\varphi^\al$ is supported in  $\overline {\hp^\ep}\times \overline \hp\cup \overline \hp\times \overline {\hp^\ep}$, so $|H_\varphi^\al|\leq(\eta_\ep(x)+\eta_\ep(y))|H_\varphi^\al|$. Using Lemma  \ref{lem-S} d) we can further bound
\begin{equation*}
|H_\varphi^\al|\leq (\eta_\ep(x)+\eta_\ep(y))F.
\end{equation*}
This allows to estimate the term $I_1$ as follows:
\begin{equation*}
|I_1|=\bigl|\iint_{\hp\times \hp} H_\varphi^{\al}\chi_\theta\,dq_\al(x)\,d\qual(y)\bigr|\leq \iint_{\hp\times \hp} (\eta_\ep(x)+\eta_\ep(y))F\chi_\theta\,d|q_\al|(x)\,d|\qual|(y).
\end{equation*}

From the localization properties of the supports of $\eta_\ep$ and $\chi_\theta$ we observe that the function $(\eta_\ep(x)+\eta_\ep(y))F\chi_\theta$ vanishes at the boundary of $\hp\times \hp$ (even in a neighborhood of this boundary). It also vanishes at infinity because $F$ vanishes at infinity. Using  that  $|\qual|\otimes|\qual|\to\mu\otimes\mu$ in $\bm(\hp\times \hp)$ we deduce that
\begin{equation*}
\iint_{\hp\times \hp} (\eta_\ep(x)+\eta_\ep(y))F\chi_\theta\,d|\qual|(x)\,d|\qual|(y)\alto \iint_{\hp\times \hp} (\eta_\ep(x)+\eta_\ep(y))F\chi_\theta\,d\mu(x)\,d\mu(y)
\end{equation*}
so
\begin{equation}\label{I1}
\limsup_{\al\to0}|I_1|\leq \iint_{\hp\times \hp} (\eta_\ep(x)+\eta_\ep(y))F\chi_\theta\,d\mu(x)\,d\mu(y).
\end{equation}
Similarly
\begin{equation}\label{I4}
|I_4|=\bigl|\iint_{\hp\times \hp} H_\varphi\chi_\theta\,d\om(x)\,d\om(y)\bigr|
\leq \iint_{\hp\times \hp} (\eta_\ep(x)+\eta_\ep(y))F\chi_\theta\,d|\om|(x)\,d|\om|(y).
\end{equation}

We deduce from \eqref{I1234}, \eqref{I2}, \eqref{I3}, \eqref{I1} and \eqref{I4} that
\begin{multline}\label{limsup}
\limsup_{\al\to0}\bigl|\iint_{\hp\times \hp} H_\varphi^{\al}\,dq_\al(x)\,d\qual(y)-\iint_{\hp\times \hp} H_\varphi\,d\om(x)\,d\om(y)\bigr|\\
\leq \iint_{\hp\times \hp} (\eta_\ep(x)+\eta_\ep(y))F\chi_\theta\,d\mu(x)\,d\mu(y)+ \iint_{\hp\times \hp} (\eta_\ep(x)+\eta_\ep(y))F\chi_\theta\,d|\om|(x)\,d|\om|(y).
\end{multline}

We let now $\theta\to0$. We clearly have that $\chi_\theta$ converges pointwise to the characteristic function of the diagonal, denoted by $\mathbbm{1}_{x=y}$, and is uniformly bounded. By the Lebesgue dominated convergence theorem we infer that
\begin{align*}
\lim_{\theta\to0}\iint_{\hp\times \hp} (\eta_\ep(x)+\eta_\ep(y))F\chi_\theta\,d\mu(x)\,d\mu(y)
&=\iint_{\hp\times \hp} (\eta_\ep(x)+\eta_\ep(y))F(x,y)\mathbbm{1}_{x=y}\,d\mu(x)\,d\mu(y)\\
&=2\int_\hp \eta_\ep(x)F(x,x)\mu(\{x\})\,d\mu(x)\\
&=0,
\end{align*}
where we  used the Fubini theorem and Claim \ref{contmes}. We can show in the same manner that
\begin{equation*}
\lim_{\theta\to0}\iint_{\hp\times \hp} (\eta_\ep(x)+\eta_\ep(y))F\chi_\theta\,d|\om|(x)\,d|\om|(y)=0
\end{equation*}
so taking the limit $\theta\to0$ in \eqref{limsup}  implies that
\begin{equation*}
\limsup_{\al\to0}\bigl|\iint_{\hp\times \hp} H_\varphi^{\al}\,dq_\al(x)\,d\qual(y)-\iint_{\hp\times \hp} H_\varphi\,d\om(x)\,d\om(y)\bigr|=0.
\end{equation*}

This shows the convergence stated in \eqref{convfinal}. Putting together \eqref{linpart}, \eqref{ualbdrylim} and \eqref{convfinal} we obtain that the limit $\al\to0$ in \eqref{weakalE} satisfies
\begin{equation*}
\int_0^T \int_\hp \partial_t\varphi\,d\om dt + \int_0^T \iint_{\hp\times \hp} H_\varphi\,d\om(x)\,d\om(y) dt + \int_\hp \varphi(0,\cdot)dq_0 = 0.
\end{equation*}
Therefore $\omega$ satisfies the weak vorticity formulation in Definition \ref{weakE} with initial vorticity $q_0$. This concludes the proof of Theorem \ref{theo-convergence}.
\end{proof}

\section{Comments and conclusions}

We have shown that solutions of the $\alpha$-Euler equations on the half-plane, under no-slip boundary conditions, converge in the vanishing $\alpha$ limit to a weak solution of the incompressible 2D Euler equations when the initial potential vorticity is independent of $\alpha$ and a bounded Radon measure of distinguished sign in $\dot{H}^{-1}$. Several comments are in order.

First, we emphasize that, for the weak solution of the 2D Euler equations which we are producing through the vanishing $\alpha$ limit, the test functions are supported {\it in the interior of} $\hp$. This is in contrast to the weak solutions obtained by two of the authors in \cite{lopes_filho_existence_2001}, for which the test functions merely {\it vanished} at the boundary of the half-plane, but their normal derivatives were not necessarily zero. Those weak solutions were called {\it boundary coupled weak solutions} and it was shown in \cite{lopes_filho_existence_2001}, see Theorem 2, that the method of images is valid for vortex sheet weak solutions if and only if they are boundary coupled. Thus the weak solutions discussed here may not give rise to a weak solution in the full plane through the method of images. This issue is under further investigation by the authors.

Second, we comment on a significant technical difference in the proof of convergence, with respect to the proof of the Delort theorem, namely that the potential vorticities $q^\al$ are not a priori bounded in $L^\infty(\real_+;H^{-1}(\hp))$. In the case of the 2D Euler equations the approximate vorticities obeyed this bound and this led to an a priori estimate on the mass of small balls: $\displaystyle{\int_{B(x;r)} \omega^n(t,\cdot)\,dy \leq C|\log r|^{-1/2}}$ which, in turn, implied no Dirac masses in the limit. In our vanishing $\alpha$ limit we use, instead, that $q^\al \rightharpoonup \omega$, $u^\al \rightharpoonup u$, so that, by linearity, $\omega = \curl u$. Since $u \in L^2$ we find $\omega \in H^{-1}$ and, thus, no Dirac masses.

In this work we have discussed only the case of flow in the half-plane. It would be interesting to study the $\alpha \to 0$ limit in a smooth, bounded domain, with vortex sheet initial data, thereby complementing and extending the results in \cite{lopes_filho_convergence_2015,busuioc_weak_2017} to less regular initial data.

\vspace{.5cm}

\scriptsize{
\textbf{Acknowledgments.}  A.V. Busuioc and D. Iftimie thank the Franco-Brazilian Network in Mathematics (RFBM) and  H.J. Nussenzveig Lopes thanks the PICS \#288801 (PICS08111)  of the CNRS, for their financial support for the scientific visits which led to this work. M.C. Lopes Filho acknowledges the support of Conselho Nacional de Desenvolvimento Científico e Tecnol\'ogico -- CNPq through grant \# 310441/2018-8 and of FAPERJ through grant \# E-26/202.999/2017 . H.J. Nussenzveig Lopes thanks the support of Conselho Nacional de Desenvolvimento Cient\'{\i}fico e Tecnol\'ogico -- CNPq through grant \# 309648/2018-1  and of FAPERJ through grant \# E-26/202.897/2018.
D.I.  has been partially funded by the  LABEX MILYON (ANR-10-LABX-0070) of Universit\'e de Lyon, within the program ``Investissements d'Avenir" (ANR-11-IDEX-0007) operated by the French National Research Agency (ANR). A.V. Busuioc and D. Iftimie thank UFRJ for their hospitality, while H. J. Nussenzveig Lopes thanks the Universit\'e de Lyon, where part of this work was completed.
}

\bigskip

\begin{description}
\item[Adriana Valentina Busuioc] Université de Lyon, Université de Saint-Etienne  --
CNRS UMR 5208 Institut Camille Jordan --
Faculté des Sciences --
23 rue Docteur Paul Michelon --
42023 Saint-Etienne Cedex 2, France.\\
Email: \texttt{valentina.busuioc@univ-st-etienne.fr}
\item[Dragoş Iftimie] Université de Lyon, CNRS, Université Lyon 1, Institut Camille Jordan, 43 bd. du 11 novembre, Villeurbanne Cedex F-69622, France.\\
Email: \texttt{iftimie@math.univ-lyon1.fr}\\
Web page: \texttt{http://math.univ-lyon1.fr/\~{}iftimie}
\item[Milton C. Lopes Filho] Instituto de Matem\'atica,
Universidade Federal do Rio de Janeiro,
Cidade Universit\'aria -- Ilha do Fund\~ao,
Caixa Postal 68530,
21941-909 Rio de Janeiro, RJ -- Brasil. \\
Email: \texttt{mlopes@im.ufrj.br} \\
Web page: \texttt{http://www.im.ufrj.br/mlopes}
\item[Helena J. Nussenzveig Lopes] Instituto de Matem\'atica,
Universidade Federal do Rio de Janeiro,
Cidade Universit\'aria -- Ilha do Fund\~ao,
Caixa Postal 68530,
21941-909 Rio de Janeiro, RJ -- Brasil. \\
Email: \texttt{hlopes@im.ufrj.br}\\
Web page: \texttt{http://www.im.ufrj.br/hlopes}
\end{description}

\end{document}